\documentclass[11pt, pdftex]{imsart}
\usepackage{amsthm,amsmath,amssymb,epic,latexsym,multirow,ulem,dsfont}
\usepackage{yfonts}

 \usepackage{lineno}

\usepackage{graphicx}

\usepackage{fancyhdr}
\usepackage{abstract}
\usepackage[numbers]{natbib}
\RequirePackage{natbib}
\RequirePackage[colorlinks,citecolor=blue,urlcolor=blue]{hyperref}
\usepackage{comment}
\usepackage{nicefrac}

\usepackage{mathrsfs}
\usepackage{eufrak}
\usepackage{wrapfig}
\usepackage[english]{babel}
 \usepackage[T1]{fontenc} 
 \usepackage[ddmmyyyy,hhmmss]{datetime}
\usepackage[svgnames,dvipsnames]{xcolor}
 \usepackage{makeidx}
\usepackage{epsf}
\usepackage[dvips]{epsfig}
\usepackage{subfig}
 \usepackage{graphicx}

\theoremstyle{plain}
\newtheorem{theo}{Theorem}[section]
\newtheorem*{theo*}{Theorem}
\newtheorem{cor}[theo]{Corollary}
\newtheorem{prop}[theo]{Proposition}
\newtheorem{lem}[theo]{Lemma}

\newtheorem*{defi*}{Definition}
\newcommand{\hZ}{\mathcal{Z}}
\newcommand{\hX}{\mathcal{X}}
\newcommand{\hY}{\mathcal{Y}}
\newcommand{\N}{\mathbb{N}}
\newcommand{\R}{\mathbb{R}}

\newcommand{\Z}{\mathbb{Z}}
\newcommand{\E}{\mathbb{E}}
\renewcommand{\P}{\mathbb{P}}

\newcommand{\lo}{\mathscr{L}}

\newcommand{\Gs}{\mathcal{G}_{1/2}}

\newcommand{\Zea}[1]{\oZ_{\rho_{#1}}}
\newcommand{\Zeea}[1]{\oZ_{\eta_{#1}}}

\newcommand{\un}{\mathds{1}}

\newcommand{\tY}{\tilde Y}
\newcommand{\nY}{Y}
\newcommand{\falf}{\cun}
\newcommand{\galf}{g_{\alpha}}
\newcommand{\ox}{x} 
\newcommand{\oX}{{X}}
\newcommand{\onY}{{\overline{{Z}}}}
\newcommand{\oXs}{\overline{X}} 
\newcommand{\oz}{\overline{z}} 
\newcommand{\oZ}{\overline{Z}} 
\newcommand{\oy}{\overline{y}} 
\newcommand{\oY}{\overline{Y}} 
 
\newcommand{\oS}{\overline{S}} 
 
\newcommand{\nZ}{Z}

\newcommand{\co}{ {\mathbf{c_{_0}}}}

\newcommand{\cun}{ {\mathbf{c_{_1}}}}
\newcommand{\cde}{ {\mathbf{c_{_2}}}}

\newcommand{\ctr}{ {\mathbf{c}}}

\newtheorem{theo2}{Theorem}

\begin{document}
\title{Axis-Driven Random Walks on $\Z^2$ \\ (transient cases)}



\begin{center}
{{\author{\fnms{Pierre} \snm{Andreoletti}}}}
\end{center}

\address{
Institut Denis Poisson,  UMR C.N.R.S. 7013, Orl\'eans, France
}
\runauthor{Andreoletti}

\begin{abstract}
Axis-driven random walks were introduced by P. Andreoletti and P. Debs \cite{AndDeb3} to provide a rough description of the behaviour of a particle trapped in a localized force field.
In contrast to their work, we examine the scenario where a repulsive force (controlled by a parameter $\alpha$) is applied along the axes, with the hypothesis that the walk remains diffusive within the cones. This force gradually pushes the particle away from the origin whenever it encounters an axis. We prove that even with a minimal force (i.e., a small $\alpha$), the walk exhibits transient, superdiffusive behaviour, and we derive the left and right tails of its distribution.
\end{abstract}





\section*{Introduction} 
In this paper, we explore a random walk on  $\mathbb{Z}^2$  whose distribution varies depending on the particle’s location. This study falls within the scope of inhomogeneous random walks, an area with extensive literature that includes a variety of models. These include, for example, oscillating walks (see \cite{Kemperman} for the pioneering result in one dimension, \cite{VoPeigne} for recent advances, and \cite{KloWoe} for extensions to higher dimensions), as well as random walks in static or dynamic random environments (see \cite{Zeitouni} and \cite{Blondel}).  A common characteristic of these models is that they completely alter the distribution of transition probabilities in comparison to a simple random walk. 
In the present work we are interested in a case where a small number of coordinates are affected by this inhomogeneity, and for which the distribution of the walk  depends on its  distance to the origin. Following the idea of the \cite{AndDeb3}, we assume that the  distribution of the random walk is only changed on some boundaries which are in our case the axes. This has for consequence that the problem also relies on random walk on the quarter plane and conditioned to remain on a cone (see for example \cite{Fayolle} ,  \cite{Denwac} and \cite{RasTar}). The model we present here is inspired by the one introduced by \cite{AndDeb3}, which itself is a toy model for a particle subjected to a magnetic field applied along two directions (the axes) with the aim of concentrating this particle in the vicinity of the origin of the lattice. Here, the problem is inverted: the force is applied from the origin outward, and only on the axes, while the particle continues to diffuse simply (like a simple random walk on $\mathbb{Z}^2$) within each of the cones.

One of the questions is then as follows: does a  small modification of the distribution on an axis break the recurrent and diffusive nature of the walk or not ? If so, can the behaviour of the walk be completely determined ? 

\noindent Let $\nZ = (\nZ_n, n \in \mathbb{N})$ represent this nearest-neighbour random walk on $\mathbb{Z}^2$. We assume that when $\nZ$ is located at a vertex on one of the axes, it has a preference to remain on the axis and moves further away from the origin with a probability governed by a parameter $\alpha$. Conversely, when the walk is inside a cone, it behaves like a simple random walk. Additionally, we assume that $\nZ$ starts at the coordinate $(1, 1)$. 

We simplify now the model by considering only the first quadrant, and we briefly discuss our expectations for the entire lattice at the end of the paper. Specifically, ${\nZ}$ is a simple random walk on $K :=\{(x, y) \in \mathbb{Z}_+^2 \mid xy \neq 0\}$, meaning that the probabilities of transition writes  $p(x, x \pm e_i) = \frac{1}{4}$ for all $x \in K$, where $e_i$ is a vector of the canonical basis. However, on $K^c \setminus \{(0, 0)\}$, the walk moves away from the origin in the following way : let $\alpha > 0$, for all $i > 0$,
\begin{align*}
   & p((i, 0), (i - 1, 0)) = p((0, i), (0, i - 1))=0, \\
   & p((i, 0), (i,  1)) = p((0, i), ( 1, i)) = {\frac{1}{2i^{\alpha}}}, \\
   & p((i, 0), (i +1, 0)) = p((0, i), (0, i + 1)) = 1 - \frac{1}{2i^{\alpha}},
\end{align*}
\begin{figure}[h] 
    \centering
    \includegraphics[width=0.5\textwidth]{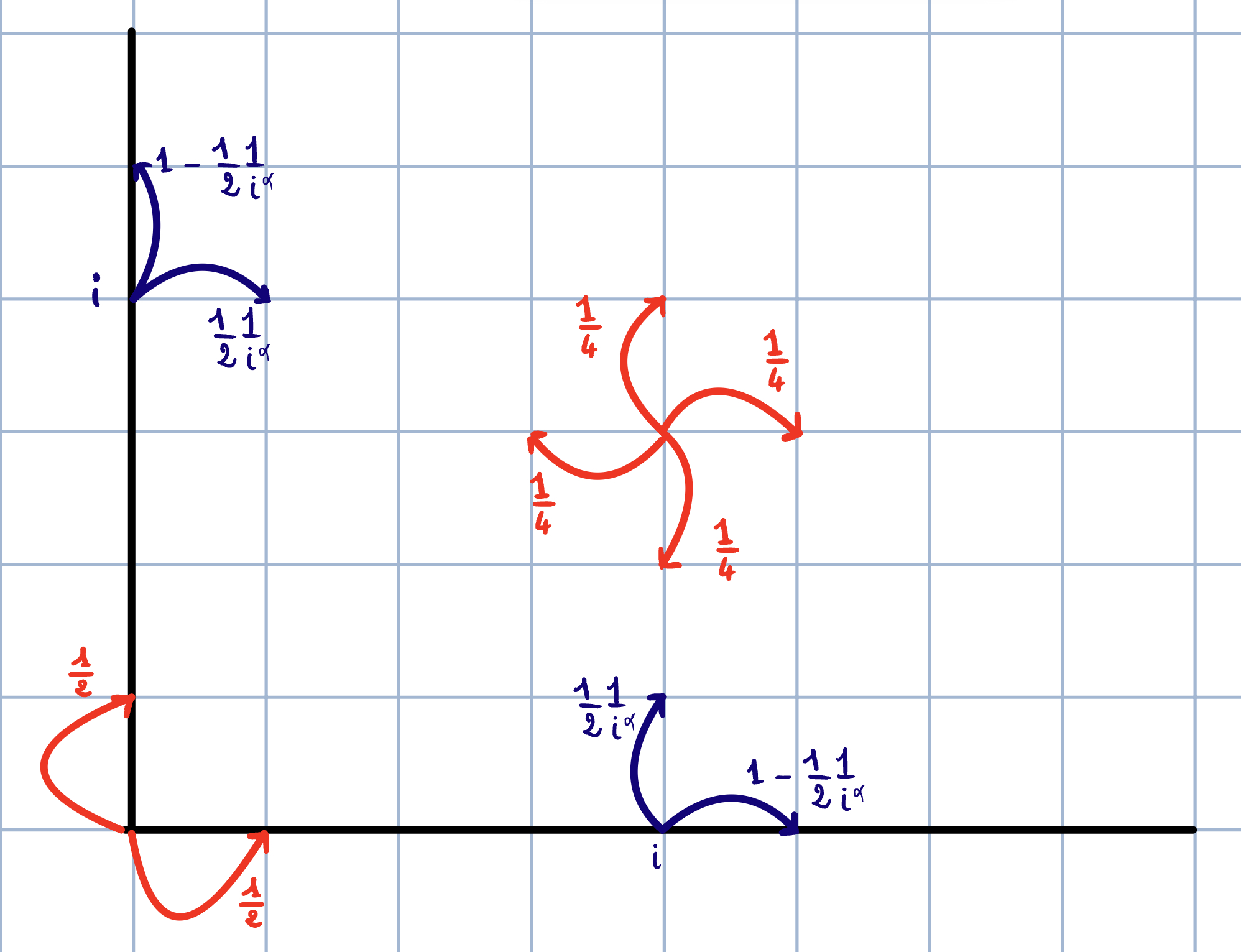}
    \caption{Transition probabilities of $Z$}
    \label{fig0}
\end{figure}
 finally, when the walk is at $(0, 0)$, it moves to either $(0,1)$ or $(1,0)$, each with probability $\frac{1}{2}$ (see also Figure \ref{fig0}). Notice that, besides the force direction, there exists another nuanced discrepancy compared to the random walk described in \cite{AndDeb3}. Specifically, we restrict the walk from retracing its steps back to the origin when it resides on the axis. While this difference does not significantly affect the behaviour of the walk, it simplifies the computations (see the brief discussion on this subject in the last section of the paper).

\noindent \\
In the sequel, we denote, for any $z=(x,y) \in \Z_+^2$, and any process $U$, $\P_z()=\P(|U_0=z)$ and we omit $z$ when $z=(1,1)$, we do the same for the expectation $\E_z()$.
The main result of this paper concerns the dominant coordinate of $Z$ at the instant $n$, that is ${\onY}_n$ where for every $z=(x,y) \in \Z^2$, $\oz:= \max(x,y)$ :   
\begin{theo2} \label{the} Assume $0<\alpha<\nicefrac{1}{2}$, for any small $a>0$  \[
    \lim_{n \rightarrow +\infty} \P({\onY}_n \leq a n^{\frac{1}{2(1-\alpha)}}) = \cde  \Big(\frac{a}{\cun}\Big)^{\frac{1-\alpha}{2}}, \]
where $\cun=(2(1-\alpha))^{(1-\alpha)^{-1}}$, $\cde=\nicefrac{8}{\sqrt{\pi}}$, and 
\[
    \lim_{n \rightarrow +\infty} \P({\onY}_n \geq a^{-1} n^{\frac{1}{2(1-\alpha)}}) = \Gs\Big(\big(\frac{\cun}{a}\big)^{2(1-\alpha)}\Big), \]
where $\Gs$ is a $\nicefrac{1}{2}$-stable distribution such that  $e^{a^{-1/2}} \Gs(a) \rightarrow 0$ when $a \rightarrow 0$.
\end{theo2}

Note that as $\alpha>0$, the walk is superdiffusive and this behaviour arises from the dynamics along the axis. However the exponent $1/2$, which also appears in the normalisation of $\oZ_n$, comes from the diffusive part on the cones. Indeed, the behaviour of $\oZ_n$ is significantly influenced by the number of excursions of the walk on the axes, which is related to the diffusive part. Thus, we can say that the diffusive component, rather than retaining the particle, amplifies the divergence of the walk, which is far from intuitive.
It turns out that this number of excursions between the axes and the cone is very close to the one of the simple random walk on the half-plane. This also explains the appearance of the  $\nicefrac{1}{2}$-stable distribution. 
\\
More precisely let us introduce the following stopping times for $\nZ$, we denote $X$ the first coordinate of $Z$ and $Y$ the second, that is $\nZ=(X,Y)$, for any $i \in \mathbb{N}^*$,
\begin{align}
 \eta_i  := \inf \{k>\rho_{i-1},  X_{k}\cdot Y_{k} =0 \}, \ \rho_{i} := \inf \{k>\eta_{i},  X_{k}\cdot Y_{k} \neq 0 \},  \nonumber   
\end{align}
that is respectively the $i$-th  exit and entrance times on one axis. Also we assume  $\rho_0 :=0$ and denote $\eta:=\eta_1$. Then if $N_n:=\max\{k, \rho_k \leq n \}$ we  prove, for example,  that for any small $u>0$, $\lim_{n \rightarrow +\infty} \P({N_n} \leq n^{1/2} u) =  \cde u^{1/2}$ (which is the same as for the simple random walk on the half-plane). Conversely the random variable $\oZ_{\rho_m}$ satisfies a law of large number which normalisation is $m^{(1-\alpha)^{-1}}$. \\
The overall result proves that even with a weak repulsion parameter ($\alpha$ can be taken as small as we want), the walk is strongly influenced by its distribution on the   axes. In particular, comparing to the simple random walk on the first quadrant with reflecting boundary on the axes, it is not diffusive and no longer recurrent. This latter fact comes from the observation that the only possibility for recurrence is that the walk can reach back the origin by diffusing on the cones. Unfortunately, diffusive behaviour alone is insufficient to counteract the superdiffusive behaviour caused by the actions along the axes (see figure \ref{fig1} at the end of the paper for a visual representation of a typical trajectory). \\
Let us conclude with a final remark on the behaviour of $Z$. Our main theorem focuses on $\oZ$, as this is the variable that exhibits the key characteristics under consideration. However, during the proof of this theorem (specifically in the proof of Lemma \ref{lem1D}), we also demonstrate that the time spent by the walk on the axis remains negligible compared to $n$, similar to the case of a simple random walk within a cone. This shows that inside the cone, the walk still performs a diffusive behaviour. This means in particular that the correct normalisation for $\min(X_n,Y_n)$ should be $\sqrt{n}$.
This also proves that axis-driven random walks exhibit great behavioural richness. Indeed, we observe both superdiffusive behaviour here, while in \cite{AndDeb4}, a recurrent and sub-diffusive case is shown to appear. And this diversity of behaviours arises even though we have only studied a small number of cases.

\medskip
The paper is organized as follows. In Section \ref{sec3}, we prove a law of large numbers for the sequence $(\oZ_{\rho_i},i)$. The proof is based on the analysis of the first and second moments of $(\oZ_{\rho_i} - \oZ_{\eta_i},i)$, whose computations are postponed to Section \ref{secmoment}. In Section \ref{secthe}, we prove the theorem using the aforementioned law of large numbers along with a study of the number of excursions of $Z$ to the axis before a given time $n$. Finally, in the last section, we discuss some potential extensions, particularly for $\alpha \geq 1/2$. We also explain why the study of this walk on the entire lattice cannot be directly derived from this work.

\section{ A law of large number \label{sec3}}

In this section, we outline the main steps of the proof regarding the behaviour of the sequence $(\oZ_{\rho_i},i)$.
This sequence is actually a sub-martingale with respect to its natural filtration  which diverges almost surely to infinity when $i \rightarrow+\infty$ (see Section \ref{seqmoment}). 
To obtain the more precise result presented below, we consider two main facts. First, we establish that \(\E(\oZ_{\rho_i})\) (almost) satisfies a recurrence relation, which allows us to derive the asymptotic behaviour of \(\E(\oZ_{\rho_i})\) (also the subject of Section \ref{seqmoment}). Additionally, we examine the behaviour of the covariance of the sequence $(\oZ_{\rho_i} - \oZ_{\eta_i},i)$ (discussed in Section \ref{SecCov}).

\noindent The main result of this section is the following


\begin{prop} \label{prop1} Assume $0<\alpha<1/2$, when $i\rightarrow +\infty$ in probability
    \begin{align*}
    \frac{\oZ_{\rho_i}}{i^{\frac{1}{(1-\alpha)}}} \rightarrow \cun,
    \end{align*}
  where $\cun=(2(1-\alpha))^{1/(1-\alpha)}$.  
\end{prop}

\noindent We decompose $\oZ_{\rho_i}$  as follows
\begin{align} \oZ_{\rho_i} = \sum_{j=1}^i (\oZ_{\rho_j} - \oZ_{\eta_j}) + \sum_{j=1}^i (\oZ_{\eta_j} - \oZ_{\rho_{j-1}}). 
\end{align} This decomposition relies on the fact that $\oZ_{\rho_i}$ is mainly driven by the part along the axis, that is, $\sum_{j=1}^i (\oZ_{\rho_j} - \oZ_{\eta_j})$. The proof of the proposition is then primarily based on the study of asymptotic behaviour for large $i$ of $\sum_{j=1 }^{ i} \E (\oZ_{\rho_j} - \oZ_{\eta_j})$ and the estimation of the covariance of the sequence $(\oZ_{\rho_j} - \oZ_{\rho_{j-1}},j)$. Also we prove that the sum $\sum_{j=1}^i (\oZ_{\eta_j} - \oZ_{\rho_{j-1}})$ is small compared to $i^{\frac{1}{1-\alpha}}$. Let us resume the main technical estimates we use here (all of them are proved in Section \ref{secmoment}),  first in Corollary \ref{Corid4} we prove the following asymptotic
   \begin{align}
    \lim_{i \rightarrow +\infty}i^{-1/(1-\alpha)} \sum_{j=1}^{ i} \E(\oZ_{\rho_j}-\oZ_{\eta_j} )    =\cun, \label{themean} 
    \end{align}
this delivers the correct asymptotic and the constant which appears in the Proposition.
We also control the sum $\sum_{j=1}^i (\oZ_{\rho_j} - \oZ_{\eta_j})$ using a second moment, this is possible because the increments $(\oZ_{\rho_j} - \oZ_{\eta_j})$ have   nice second moments (when $\alpha  <1/2$) and therefore exhibit only slight correlations. So in Proposition \ref{procov}, we prove that  
\begin{align}
     \lim_{i \rightarrow + \infty} i^{-(2\alpha+2)/(1-\alpha)} \E\Big( \big(\sum_{j=1}^{ i} (\oZ_{\rho_j}-\oZ_{\eta_{j}})-\E(\oZ_{\rho_j}-\oZ_{\eta_{j}})  \big)^2\Big) =0  . \label{Cheb} 
\end{align}

\noindent This implies that in probability $i^{-1/(1-\alpha)} \sum_{j\leq i} (\oZ_{\rho_j}-\oZ_{\rho_j})$ behaves like $i^{-1/(1-\alpha)} \sum_{j\leq i} \E(\oZ_{\rho_j}-\oZ_{\rho_j})$ which itself behaves like $\cun$ by \eqref{themean}.

Finally we prove that $\sum_{j=1}^i (\oZ_{\eta_j} - \oZ_{\rho_{j-1}})$ is negligible comparing to $i^{1/(1-\alpha)}$. For that just note that  $|\sum_{j=1}^i(\oZ_{\eta_j}-\oZ_{\rho_{j-1}})| \leq \sum_{j=1}^i|\oZ_{\eta_j}-\oZ_{\rho_{j-1}}|$ and   $\E(|\oZ_{\eta_j}-\oZ_{\rho_{j-1}}|)=\E(\E_{Z_{\rho_{j-1}}}(|\oZ_{\eta}-\oZ_{\rho_{j-1}}||Z_{\rho_{j-1}} ))\leq cste  $, where the last inequality  comes from Lemma  \ref{lemXetastar}.  This implies, by using the Markov's inequality, for any $\epsilon>0$ such that $1/(1-\alpha)-\epsilon>1$,  
\begin{align}
\P\Big( \Big|\sum_{j=1}^{ i}(\oZ_{\eta_j}-\oZ_{\rho_{j-1}})\Big| >i^{1/(1-\alpha)-\epsilon} \Big) \leq  i^{-1/(1-\alpha)+\epsilon}  \sum_{j=1}^i cste \rightarrow 0.
\end{align}

\section{Convergences in law (proof of the theorem) \label{secthe}}

This section is devoted to the proof of Theorem \ref{the} so we always assume that $0<\alpha<1/2$ in this section.   

\subsection{Preliminaries (an elementary coupling) \label{coupling}}

In this first section, we show that after a small amount of time compared to $n$, the walk has chosen an axis (horizontal or vertical) and then will not reach the other axis at least before the instant $n$. Of course, the walk still oscillate between the interior of the cone and the chosen axis (in fact only in the case $\alpha \geq 1$ the walk can stay on the axis for ever, see Section \ref{sec5}).

\begin{lem} \label{lem1D}
For any $0<\epsilon<1$, $\lim_{n \rightarrow +\infty}\P(\{  \forall j \in [ n^{1-\epsilon}, n],\ X_{j}=0\}\ or\ \{  \forall j \in [ n^{1-\epsilon}, n],\ Y_{j}=0\} )=1$.
\end{lem}

\begin{proof}
We rely on two key arguments. First, the number of excursions to the axes for this random walk behaves similarly to that of a simple random walk in the quarter plane and is of order $n^{1/2}$. Second, by Proposition \ref{prop1}, within these $n^{1/2}$ excursions, the walk can reach a coordinate at a distance $n^{1/[2(1-\alpha)]} \gg n^{1/2}$ from the origin. Consequently, since the fluctuations within the interior of the cone are of order $n^{1/2}$ (up to time $n$), the walk starting from this significantly distant coordinate ($n^{1/[2(1-\alpha)]}$) is unlikely to return to the opposite side (i.e., the other axis). We now provide further details. \\
First, we explain why the number of excursions to the axis of $Z$ before time $n$, denoted as $N_n := \max\{k > 0 : \rho_k \leq n\}$, remains almost the same as for the simple random walk on the quarter plane (with no constraint at all). For this latter random walk, the transition probabilities are the same as for $Z$ everywhere except on the  axis, where it has a probability of $\nicefrac{1}{4}$ to move forward or backward, and $\nicefrac{1}{2}$ to leave the axis.  So note first that the local time on the axis of $Z$ before the instant $n$ remains far below $n$. Indeed, for a given $i$, the local time on the axis before $\rho_i$ can be expressed as $\sum_{m=1}^i (\rho_m - \eta_m-1) = \sum_{m=1}^i (\oZ_{\rho_m} - \oZ_{\eta_m})$. Then, for any $\epsilon$, we have $\P\left(\sum_{m=1}^i (\rho_m - \eta_m-1) > i^{\frac{1}{1 - \alpha} + \epsilon}\right) \leq \E\left(\sum_{m=1}^i (\oZ_{\rho_m} - \oZ_{\eta_m})\right) i^{-\frac{1}{1 - \alpha} - \epsilon} \sim \text{const} \cdot i^{\frac{1}{1 - \alpha}} \cdot i^{-\frac{1}{1 - \alpha} - \epsilon},$ where the last equivalence comes from \eqref{id10}. So if we consider a number of excursions to the axis of order $n^{1/2}$ (this value is by definition smaller for $Z$ than for the simple random walk on the quarter plane which is itself of order $n^{1/2}$), the total amount of time spent by $Z$ during these excursions on the axis is at most of order $n^{\frac{1}{2}\left(\frac{1}{1 - \alpha} + \epsilon\right)}$. Since $\alpha < \frac{1}{2}$, $\epsilon$ can be chosen such that $n^{\frac{1}{2}\left(\frac{1}{1 - \alpha} + \epsilon\right)} = o(n)$. This implies that the local time of $Z$ during the first $n$ steps in the interior of the cone is approximately $n - o(n)$, which implies that the number of excursions to the axes (which depends on the trajectories in the interior of the cone) remains of order $n^{1/2}$, similar to the simple random walk on the quarter plane. Note that in the following section, we will provide a more precise result for $N_n$ (the number of excursion to the axes of $Z$ until the instant $n$).
\\ 
This first step shows that for any $\delta > 0$, with a probability close to one, $N_n \geq n^{1/2 - \delta/2}$. In a second step, we use Proposition \ref{prop1}, which states that with a large probability $\rho_{n^{1/2 - \delta/2}} \geq n^{(1 - \delta)/[2(1 - \alpha)]}$, moreover $\delta$ can be chosen such that $(1 - \delta)/[2(1 - \alpha)] > 1/2$. This implies that, with a probability close to one, before the instant $n$, $Z$ will reach a coordinate that is at least $n^{(1 - \delta)/[2(1 - \alpha)]}$ distant from the origin. However, since the only possibility for the random walk to move backward from this distant coordinate is to do so within the interior of the cone, and given that the walk is only diffusive within this cone, it becomes impossible for it, with a probability converging to one, to make a move of order $n^{(1 - \delta)/[2(1 - \alpha)]}$ before the instant $n$. To conclude the proof, note that with a probability converging to one, the $n^{1/2 - \delta}$ excursions to the axis will be achieved within a time frame of $n^{1 - \epsilon}$, even if it requires adjusting $\epsilon$ according to $\delta$.
\end{proof}

This last lemma tells that $Z$ choose at some point a direction, let say the horizontal axis, and from this instant and until the instant $n$ (at least) it will never reach again the other direction (the vertical axis here).  
With this in mind we introduce a new random walk $\hZ:=(\hX,\hY)$, which has the following probabilities of transition when it is on an axis (see also Figure \ref{fig2}): for any $i>0$
\begin{align*} 
& p^{\hZ}((i,0),(i+1),0)=1-p^{\hZ}((i,0),(i,1))=1-1/2i^{\alpha}, \\
& p^{\hZ}((0,i),(1,i)) =2p^{\hZ}((0,i),(0,i-1))=2p^{\hZ}((0,i),(0,i+1))=1/2, \\
& p^{\hZ}((0,0),(0,1))=p^{\hZ}((0,0),(1,0))=1/2, 
\end{align*}
and which has the simple random walk probabilities of transition ($1/4$ in every directions) in the interior of the cone. 
\begin{figure}[h] 
    \centering
    \includegraphics[width=0.5\textwidth]{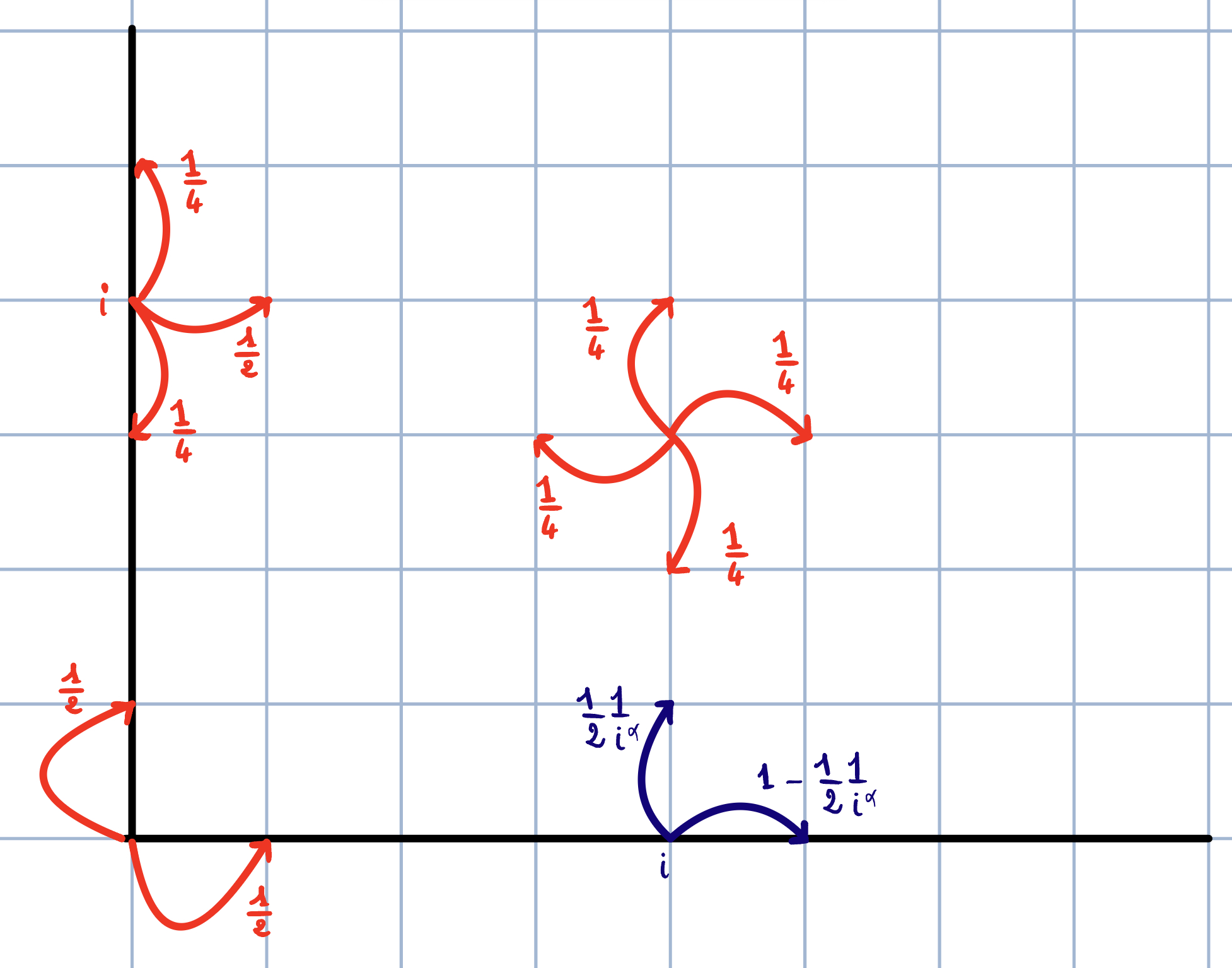}
    \caption{Transition probabilities Transition probabilities of $\hZ$}
    \label{fig2}
\end{figure}
So typically the random   walk $\hZ$ is subject to the $\alpha$-force on the horizontal axes, it has a reflecting barrier on the vertical axis and is a simple random walk elsewhere. \\
The previous result shows that we can establish a simple coupling between $Z$ and $\hZ$. This coupling is constructed as follows: for any $1 > \epsilon > 0$ and for any $k$ between the instants $n^{1 - \epsilon}$ and $n$, we have $Z_k = \hZ_k$. Such a coupling exists thanks to Lemma \ref{lem1D}. \\ 
 The next step is to obtain the asymptotic in $i$  of $\rho_i^{\hZ}:=\inf\{k>\eta_{i}^{\hZ}, \hY_k=1 \}$ where $\eta_i^{\hZ}:=\inf\{k>\rho_{i-1}^{\hZ}, \hY_k = 0 \}$, with $\rho_0^{\hZ}=0$ and $N_n^{\hZ}:=\max\{k>0, \rho_k^{\hZ} \leq n\}$. Also we denote $\eta^{\hZ} :=\eta^{\hZ}_1$.


\subsection{The number of excursions on the axes until the instant $n$}
We start with the study of the asymptotic of  $\rho_m^{\hZ}$, from which we deduce the left tail of $N_n^{\hZ}:=\max\{k>0, \rho_k^{\hZ} \leq n\}$, then we consider its right tail. 
\begin{prop} \label{loiderho2} For the Skorokhod $J_1$-topology, when $i \rightarrow +\infty$
    \begin{align}
    (\frac{{\rho^{\hZ}_{\lfloor t i \rfloor }}}{i^{2}},t) \xrightarrow{\mathcal{L}} (V_{1/2}(t),t),
    \end{align}
 $V_{1/2}(t)$ is 1/2-stable Lévy subordinator with Lévy triple $(0,0,v_{1/2})$ with $v_{1/2}(dt)= \cde t^{-3/2}\un_{t>0}dt$ and recall $\cde=8/ \sqrt{\pi}$.
\end{prop}

\begin{proof}

Let us decompose $\rho_i^{\hZ}$ as $\rho_i^{\hZ} = \sum_{j=1}^i (\rho_j^{\hZ} - \eta_j^{\hZ}) + \sum_{j=1}^i (\eta_j^{\hZ} - \rho_{j-1}^{\hZ})$. By symmetry and the strong Markov property, the random variables $(\eta_j^{\hZ} - \rho_{j-1}^{\hZ})_j$ are independent and have the same law as the random variable  $\eta^{\hZ} $, 
for which  uniformly in $z=(\ox,1)$, for any $u > 0$
\begin{align}
    \lim_{m \rightarrow +\infty} m^{1/2}  \P_z(\eta^{\hZ} >m u)= \cde u^{-1/2}. \label{tale_eta_star}
\end{align}  
So $ \sum_{j=1}^i \eta_j^{\hZ} - \rho_{j-1}^{\hZ}$ is equal in law to $H_i:=\sum_{j=1}^i \eta^{(j)}$, where for any $j$,  $\eta^{(j)}$ is an independent copy of  $\eta^{\hZ} $, so we have the convergence in law with respect to the $J_1$-topology, of the process  $ (\frac{{H_{\lfloor t i \rfloor }}}{i^{2}},t)$ toward $(V_{1/2}(t),t)$. 
We are left to prove that the sum $\sum_{j=1}^i (\rho_j^{\hZ}-\eta_j^{\hZ})$ is negligible comparing to $i^2$,  by definition $\sum_{j=1}^i (\rho_j^{\hZ}-\eta_j^{\hZ}-1)=\sum_{j=1}^i (\hX_{\rho_j}-\hX_{\eta_j})$ moreover this random variable is stochastically dominated by  $\sum_{j=1}^i (\oZ_{\rho_j}-\oZ_{\eta_j})$ and we know from \eqref{id10} that the mean of this last random variable is of order $i^{1/(1-\alpha)}$, this implies that in probability $\sum_{j=1}^i (\oZ_{\rho_j}-\oZ_{\eta_j})$ and therefore $\sum_{j=1}^i (\rho_j^{\hZ}-\eta_j^{\hZ})$ is smaller than $i^{1/(1-\alpha)+ \epsilon}$ for an $\epsilon>0$ as small as needed, but as $\alpha<1/2$ we can choose $\epsilon$ in such a way that  $i^{1/(1-\alpha)+ \epsilon}=o(i^2)$. \end{proof}

\noindent In the corollary below we state the left  and right tails for the number of excursions $N_n^{\hZ}$ before the instant $n$. 

\begin{cor} \label{cor2} For any small $u>0$, $\lim_{n \rightarrow +\infty} \P(N_n^{\hZ} \leq n^{1/2} u) = \cde u^{1/2}$ and $\lim_{n \rightarrow +\infty} \P(N_n^{\hZ} \geq n^{1/2} u^{-1}) =  \Gs(u)$, where $\cde$ and $\Gs$ is defined in Theorem \ref{the}.
\end{cor}
\begin{proof}
For the first equality, note that $\P(N^{\hZ}_n \leq n^{1/2} u) = \P(\rho_{\lfloor n^{1/2} u \rfloor}^{\hZ} \geq n)$ and apply Proposition \ref{loiderho2}. For the second, recall that in the proof of Proposition \ref{loiderho2}, we showed that for any $i$, $\rho_i^{\hZ}$ can be approximated by $\sum_{j=1}^{i} (\eta_j^{\hZ} - \rho_{j-1}^{\hZ})$. Moreover, this sum is equal in law to $H_i := \sum_{j=1}^{i} \eta^{(j)}$, where $(\eta^{(j)},j)$ are independent copies of $\eta^{\hZ} $. Now let $N_n^* := \sup\{i > 0, \sum_{j=1}^{i} \eta^{(j)} \leq n\}$. A usual result (see (5.1) in chapter XI.5 and Theorem 2 in chapter XIII.6 of \cite{Feller2}) shows that $\P(N_n^* \geq n^{1/2} u^{-1}) \rightarrow \Gs(u)$. This implies the result for $N_n^{\hZ}$. 

\end{proof}

\subsection{Convergence of  the random walks $\hZ$}
In this section, we prove a result very similar to Theorem \ref{the} but for the coordinate $\hX$ of $\hZ$, then by the way of the coupling discussed at the end of Section \ref{coupling} this leads directly to Theorem \ref{the}.

\begin{prop} \label{prop3.4} For any small $a>0$,    \[
    \lim_{n \rightarrow +\infty} \P(\hX_n \leq a n^{\frac{1}{2(1-\alpha)}}) = \cde  (a/\cun)^{{(1-\alpha)/2}}, \]
and 
\[
    \lim_{n \rightarrow +\infty} \P(\hX_n \geq a^{-1} n^{\frac{1}{2(1-\alpha)}}) = \Gs\Big(\big(\frac{\cun}{a}\big)^{2(1-\alpha)}\Big). \]
\end{prop}


\begin{proof}
\noindent Let $\epsilon>0$ introduce 
\begin{align*}
& C_{\epsilon,n}:=\{|\hX_n-\hX_{\rho^{\hZ}_{N_n^{\hZ}}}|>\epsilon n^{{(2(1- \alpha))^{-1}}} \}, \\
& D_{\epsilon,n}:=\{|\hX_{\rho^{\hZ}_{N_n^{\hZ}}}-\cun (N_n^{\hZ})^{{(1- \alpha)^{-1}}}| > \epsilon \cun (N_n^{\hZ})^{{(1- \alpha)^{-1}}}\}.
\end{align*}
Assume for the moment that the  following result is true : $\lim_{n \rightarrow +\infty}\P(C_{\epsilon,n}\cup D_{\epsilon,n})=0$. This implies that the last excursion before $n$ will not affect the normalisation of $\hX_n$  and that $\hX_{\rho^{\hZ}_{N_n^{\hZ}}}$ can be approximated by $\cun (N_n^{\hZ})^{{(1- \alpha)^{-1}}}$. We start with the left tail :
\begin{align*}
    \P(\hX_n \leq a n^{{(2(1- \alpha))^{-1}}} ) & \leq  \P(\hX_n \leq a n^{{(2(1- \alpha))^{-1}}}, C_{\epsilon,n}^c)+o(1)  \leq   \P(\hX_{\rho^{\hZ}_{N_n^{\hZ}}} \leq a n^{{(2(1- \alpha))^{-1}}}(1+ \epsilon) )+o(1) \\
   &=\P(\hX_{\rho^{\hZ}_{N_n^{\hZ}}} \leq a n^{{(2(1- \alpha))^{-1}}}(1+ \epsilon),D_{\epsilon,n}^c  ) +o(1)  \\
   & \leq \P( \cun (N_n^{\hZ})^{{(1- \alpha)^{-1}}} (1- \epsilon) \leq a n^{{(2(1- \alpha))^{-1}}}(1+ \epsilon)  ) +o(1) \\
   &= \P( N_n^{\hZ}/n^{1/2}  \leq (a/\cun)^{1-\alpha } ((1+\epsilon)/(1- \epsilon))^{(1-\alpha) }    ) +o(1)
\end{align*}
Using the left tail in Corollary \ref{cor2}, we obtain the upper bound. The lower bound is obtained similarly with no additional difficulty. The proof of the right tail follows the same line and uses the right tail of Corollary \ref{cor2}.
\end{proof}

\noindent To finish we prove separately that both probabilities $\P(C_{\epsilon,n})$ and $\P(D_{\epsilon,n})$ converge to 0. For that we need the following preliminary result :
\begin{lem} \label{lem1} There exists $c>0$ such that, for any $\epsilon>0$, $\lim_{n \rightarrow +\infty} \P(E_{\epsilon,n}:=\{n-\rho_{N_n^{\hZ}}^{\hZ} \leq \epsilon n\})=c \epsilon^{1/2}$. 
\end{lem}

\begin{proof}
We have already manipulated the idea of the proof in both Proposition \ref{loiderho2} and Corollary \ref{cor2}. In the proof of Proposition \ref{loiderho2} we have showed that for any $i$, $\rho_i^{\hZ}$ can be approximated by $\sum_{j=1}^{i}(\eta_j^{\hZ}-\rho_{j-1}^{\hZ})$, moreover this sum is equal in law to  $H_i:=\sum_{j=1}^{i}\eta^{(j)}$ where $(\eta^{(j)},j)$ are independent copies of $\eta^{\hZ} $ such that for any $b>0$ and large $m$,   $\P(\eta>b m)\sim \cde (b m)^{-1/2}$. Also recall that $N_n^*:=\sup\{i>0, \sum_{j=1}^{i}\eta^{(j)}\leq n \}$, then a  usual result (the arcsinus law, see for example \cite{Feller2} chapter XIV.3, page 447) tells that $\lim_{n \rightarrow +\infty} \P(n-H_{ N_n^*}  \leq \epsilon n) =\frac{1}{\pi} \int^{\epsilon}_0 u^{-1/2}(1-u)^{-1/2}du= c\epsilon^{1/2}.$ This proves the Lemma.
\end{proof}

We now prove that the event $C_{\epsilon,n}$ is unlikely to occur.

\begin{lem} \label{lem3}  $\lim_{\epsilon \rightarrow 0} \lim_{n \rightarrow +\infty} \P(C_{\epsilon,n})=0$.
\end{lem}

\begin{proof}
 First we have to ensure that $\hX_{\rho_{N_n^{\hZ}}}$ is not too large, that is to say smaller than $ n^{1/2(1-\alpha)^{-1}+\epsilon}$ for any positive $\epsilon$. This fact is actually a consequence of Proposition \ref{prop1},  Corollary \ref{cor2} and Lemma \ref{lem1}. Let us give some details : by Corollary \ref{cor2} and Lemma \ref{lem1} 
\begin{align*}
    &\P(\hX_{\rho^{\hZ}_{N_n^{\hZ}}} > n^{1/2(1-\alpha)^{-1}+\epsilon}) =\\ 
    &\sum_{ j \leq A n^{1/2}} \sum_{x, \ox \geq n^{1/2(1-\alpha)^{-1}+\epsilon}} \sum_{k \leq n-\epsilon n} \P( \hX_{\rho_{j}^{\hZ}}=x, \rho_{j}^{\hZ}=k)\P_z( \rho_{1}^{\hZ}>n-k)  +o(1)+c \epsilon^{1/2},
\end{align*}
 where $z=(x,1)$. For any $k\leq n-\epsilon n$, the tail $\P_z (\rho_{1}^{\hZ}>n-k)$ is supported by the first exit time from the cone (the part on the axis is actually  exponentially decreasing, see the distribution of $\rho$ at the beginning of the proof of Lemma \ref{lem2.4} and $n-k \geq \epsilon n$), so $\P_x (\rho^{\hZ}_{1}>n-k) \sim \P_x (\eta^{\hZ} >n-k)\leq C(n-k)^{-1/2}$ (see \eqref{tale_eta_star}) for a given constant $C$ : 
\begin{align*}
    &\P(\hX_{\rho^{\hZ}_{N_n^{\hZ}}} > n^{1/2(1-\alpha)^{-1}+\epsilon})\\ 
    &=\frac{cst}{(\epsilon n)^{1/2}}\sum_{ j \leq A n^{1/2}}  \sum_{k \leq n-\epsilon n} \P(\hX_{\rho_{j}^{\hZ}} > n^{1/2(1-\alpha)^{-1}+\epsilon} , \rho^{\hZ}_{j}=k)  +o(1)+c \epsilon^{1/2}\\
    & \leq \frac{cst A}{(\epsilon)^{1/2}} \max_{j \leq A n^{1/2}}\P(\hX_{\rho^{\hZ}_{j}} > n^{1/2(1-\alpha)^{-1}+\epsilon}) +o(1)+c \epsilon^{1/2}
\end{align*}
which converges to zero when $n$ goes to infinity by Proposition \ref{prop1} (we use that $\hX$ is stochastically dominated by $\oZ$). We are now ready to prove the lemma. \\
By Corollary \ref{cor2}, Lemma \ref{lem1} and the above upper bound for $\hX_{\rho^{\hZ}_{N_n^{\hZ}}}$
\begin{align*}
     \P(C_{\epsilon,n}) 
     & \leq \sum_{j=1}^{\frac{1}{\epsilon}n^{1/2}} \sum_{x, \ox \leq n^{1/2(1-\alpha)^{-1}+\epsilon}} \sum_{k \leq n-\epsilon n} 
      \P( \hX_{\rho^{\hZ}_{j}}=x, \rho^{\hZ}_{j}=k) \\ & \times \P_z (|\hX_{n-k}-x|>\epsilon n^{{(2(1- \alpha))^{-1}}}, \rho^{\hZ}_{1}>n-k) + c\epsilon^{1/2}+o(1).
\end{align*}
We decompose the last probability into two, in one part we assume that for any $0\leq m \leq n-k$, $\hZ_{m}$ is in the cone whereas the other one that $\hZ_{n-k}$ belongs to the horizontal axis  : 
\begin{align*}
& \P_z (|\hX_{n-k}-\ox|>\epsilon n^{{(2(1- \alpha))^{-1}}}, \rho_{1}^{\hZ}>n-k) \\
& \leq  \P_z (|\hX_{n-k}-\ox|>\epsilon n^{{(2(1- \alpha))^{-1}}}, \eta^{\hZ}> n-k ) \\
& + \P_z (|\hX_{n-k}-\ox|>\epsilon n^{{(2(1- \alpha))^{-1}}}, \eta^{\hZ} \leq  n-k, \rho_{1}^{\hZ}>n-k).
\end{align*}
For the first probability, under $\eta> n-k$, $(\hX_m,m \leq n-k)$ is a simple random walk on the cone. Moreover as $\alpha>0$ and as we work under the probability measure $\P_z$ we ask this simple random walk to make a fluctuation larger than $ n^{(2(1- \alpha))^{-1}}$ which is way greater than $(n-k)^{1/2}$, so there exists $c'$ such that   
\begin{align} 
\P(|\hX_{n-k}-\ox|>\epsilon n^{{(2(1- \alpha))^{-1}}}, \eta^{\hZ}> n-k ) \leq e^{-c'\epsilon^2 n^{{(1- \alpha)^{-1}}} (n-k)^{-1} } \leq e^{-c'\epsilon^2 n^{\alpha/(1- \alpha)}}. \label{eq7}
\end{align}

To address the second probability, it is important to observe that for the walk (under the event \{$\eta^{\hZ} \leq  n-k, \rho_1^{\hZ}>n-k-l$\}) to exhibit a behaviour like  $\hX_{n-k}-\ox < -\epsilon n^{{(2(1- \alpha))^{-1}}}$, this has to be during the excursion on the cone, because indeed the walk is increasing on the axis. So similarly as above 
\begin{align}
 &\P_z (\hX_{n-k}-\ox<-\epsilon n^{{(2(1- \alpha))^{-1}}}, \eta^{\hZ} \leq  n-k, \rho^{\hZ}_{1}>n-k)  \leq e^{-c'\epsilon^2 n^{\alpha/(1- \alpha)}}. \label{eq8}
\end{align}
The last event to be treated $\{\hX_{n-k}-\ox>\epsilon n^{{(2(1- \alpha))^{-1}}} \}$ is some more delicate.
It seems that we have no choice but to decompose this event based on the values of the pair $(\eta^{\hZ},\hX_{\eta^{\hZ}})$
\begin{align}
 &\P_z (\hX_{n-k}-\ox>\epsilon n^{{(2(1- \alpha))^{-1}}}, \eta^{\hZ} \leq  n-k, \rho^{\hZ}_{1}>n-k) \nonumber \\
 &=\sum_{l=1}^{n-k-1} \sum_y \P_z(\eta^{\hZ}=l,\hX_{\eta^{\hZ}}=y)\P_{(y,0)} (\hX_{n-k-l}-\ox>\epsilon n^{{(2(1- \alpha))^{-1}}}, \rho^{\hZ}>n-k-l), \label{eq10}
\end{align}
where $\rho^{\hZ}=\inf\{k>0, \hY_k =1\}$.
We decompose the sum on $y$ into two, either $y$ is such that $y-\ox> \epsilon n^{{(2(1- \alpha))^{-1}}}/2$ or not. In the case  $y-\ox> \epsilon n^{{(2(1- \alpha))^{-1}}}/2$ (which implies that the effort is carried out by the walk on the cone), we have again
\begin{align}
&  \sum_{l=1}^{n-k-1} \sum_{y,y-\ox> \epsilon n^{{(2(1- \alpha))^{-1}}}/2 } \P_z(\eta^{\hZ}=l,\hX_{\eta^{\hZ}}=y)   \P_{(y,0)} (\hX_{n-k-l}-\ox>\epsilon n^{{(2(1- \alpha))^{-1}}}, \rho^{\hZ}>n-k-l) \nonumber \\
 & \leq \sum_{l=1}^{n-k-1} \sum_{y,y-\ox> \epsilon n^{{(2(1- \alpha))^{-1}}}/2 } \P_z(\eta^{\hZ}=l,\hX_{\eta^{\hZ}}=y)  \leq \P_z(\hX_{\eta^{\hZ}}-x>\epsilon n^{{(2(1- \alpha))^{-1}}}/2,\eta^{\hZ} \leq n-k-1) \nonumber \\
 &\leq  e^{-c'\epsilon^2 n^{\alpha/(1- \alpha)}}. \label{eq11}
\end{align}
where the last inequality comes, as before, from the fact that the probability for a diffusion process to fluctuate more than the square root of its time is exponentially decreasing. Otherwise when $y-\ox \leq  \epsilon n^{{(2(1- \alpha))^{-1}}}/2$ (here the effort is carried out by the trajectory on the axis) :   
\begin{align*}
&  \sum_{l=1}^{n-k-1} \sum_{y,y-\ox< \epsilon n^{{(2(1- \alpha))^{-1}}}/2 } \P_z(\eta^{\hZ}=l,\hX_{\eta^{\hZ}}=y)   \P_{(y,0)} (\hX_{n-k-l}-\ox>\epsilon n^{{(2(1- \alpha))^{-1}}}, \rho^{\hZ}>n-k-l)\\
 & \leq  \sum_{y,\oy-\ox< \epsilon  n^{{(2(1- \alpha))^{-1}}}/2 } \P_z(\hX_{\eta^{\hZ}}=y) \P_{(y,0)} (\hX_{\rho^{\hZ}}-\ox>\epsilon n^{{(2(1- \alpha))^{-1}}}), 
\end{align*}
where the last inequality comes from the fact that the walk is increasing on the axis. We now refer to the proof of Lemma \ref{lem2.4} for the distribution of $\hX_{\rho^{\hZ}}$ (the distribution of $\hX_{\rho^{\hZ}}$ is actually the same than the one of $\oZ_{\rho}$), we obtain
\begin{align*}
\P_{(y,0)} (\hX_{\rho^{\hZ}}-\ox>\epsilon n^{{(2(1- \alpha))^{-1}}}) & =\prod_{m=y}^{\ox+ \epsilon n^{1/(2(1-\alpha))}}(1-m^{-\alpha}/2)  \leq \prod_{m=\ox+ \epsilon n^{1/(2(1-\alpha))}/2}^{\ox+ \epsilon n^{1/(2(1-\alpha))}}(1-m^{-\alpha}/2) \\
& \sim e^{-\frac{1}{2(1- \alpha)}((\ox+ \epsilon n^{1/(2(1-\alpha))})^{1- \alpha}-(\ox+ \epsilon n^{1/(2(1-\alpha))}/2)^{1- \alpha})}
\end{align*}
Then if $\ox \leq  \epsilon n^{1/(2(1-\alpha))}$, there exists a constant $c>0$ such that $\P_{(y,0)} (\hX_{\rho^{\hZ}}-\ox>\epsilon n^{{(2(1- \alpha))^{-1}}}) \leq e ^{-c \epsilon \ox  n^{\alpha/(2(1-\alpha))}}$, otherwise  $\P_{(y,0)} (\hX_{\rho^{\hZ}}-\ox>\epsilon n^{{(2(1- \alpha))^{-1}}}) \sim e^{- \frac{\epsilon}{2}\ox^{-\alpha} n^{(2(1-\alpha))^{-1}}}$ and as $\ox \leq n^{{(2(1- \alpha))^{-1}}+\epsilon}$,  $\P_{(y,0)} (\hX_{\rho^{\hZ}}-\ox>\epsilon n^{{(2(1- \alpha))^{-1}}}) \leq e^{- n^{1/2-\alpha \epsilon}}$.  So $\sum_{y,y-\ox< \epsilon  n^{{(2(1- \alpha))^{-1}}}/2 } \P_z(\hX_{\eta^{\hZ}}=y) \P_{(y,0)} (\hX_{\rho^{\hZ}}-\ox>\epsilon n^{{(2(1- \alpha))^{-1}}}) \leq e^{- n^{1/2-\alpha \epsilon}}$. This last upper bound together with \eqref{eq11} and  \eqref{eq10}, yields that there exists a constant $c'$ such that $\P_z (\hX_{n-k}-\ox>\epsilon n^{{(2(1- \alpha))^{-1}}}, \eta^{\hZ} \leq  n-k, \rho_{1}^{\hZ}>n-k) \leq e^{-n^{c'}}$. This together with \eqref{eq7} and \eqref{eq8}  back in the decomposition of $\P(C_{\epsilon,n})$, yields the Lemma.

\end{proof}

We conclude this section with a Lemma and a remark summarizing the theorem's proof.

\begin{lem}  For any $\epsilon>0$, 
   \begin{align}
    & \lim_{n \rightarrow +\infty} \P( D_{\epsilon,n})=0.  \label{eq12}
    \end{align}
\end{lem}

\begin{proof} 
\noindent \\
Like for the proof of Proposition \ref{prop1}, we decompose $\hX_{\rho^{\hZ}_{N_n^{\hZ}}}=\sum_{m=1}^{N^{\hZ}_n}(\hX_{\rho^{\hZ}_m}-\hX_{\eta^{\hZ}_{m}})+\sum_{m=1}^{N^{\hZ}_n}(\hX_{\eta^{\hZ}_m}-\hX_{\rho^{\hZ}_{m-1}})$. \\
Assume for the moment that $\Sigma_{N_n^{\hZ}}':=\sum_{m=1}^{N_n^{\hZ}}{\hX_{\eta_m^{\hZ}}-\hX_{\rho_{m-1}^{\hZ}}}=o((N_n^{\hZ})^{(1- \alpha)^{-1}})$, so we only prove that \eqref{eq12}  is still true if we replace $\hX_{\rho^{\hZ}_{N_n^{\hZ}}}$ by ${\sum_{m=1}^{N^{\hZ}_n}(\hX_{\rho^{\hZ}_m}-\hX_{\eta^{\hZ}_{m}})}=:\Sigma_{N^{\hZ}_n}$ : let $A>0$ large and introduce the interval $J_n:=(A^{-1}n^{1/2}, A n^{1/2})$, let $\epsilon>0$ using Lemma \ref{lem1} and Corollary \ref{cor2} and strong Markov property
   \begin{align*}
    & \P(|\Sigma_{N_n^{\hZ}}-\cun (N_n^{\hZ})^{(1- \alpha)^{-1}}| > \epsilon \cun (N_n^{\hZ})^{(1- \alpha)^{-1}}) \\
     \leq  &
    \P(|\Sigma_{N_n^{\hZ}}-\cun (N_n^{\hZ})^{(1- \alpha)^{-1}}| > \epsilon \cun (N_n^{\hZ})^{(1- \alpha)^{-1}},E_{\epsilon,n}, N_n^{\hZ} \in J_n) +O(\epsilon^{1/2})+O(A^{-1/2}) \\
    =& \sum_{ j \in J_n} \sum_x \sum_{k \leq n-\epsilon n} \P(|\Sigma_{j}-\cun j^{(1- \alpha)^{-1}}| > \epsilon \cun j^{(1- \alpha)^{-1}}, \hX_{\rho_{j}^{\hZ}}=x, \rho_{j}^{\hZ}=k) \times \\
    & \P_z (\rho_{1}^{\hZ}>n-k)+O(\epsilon^{1/2})+O(A^{-1/2}).
    \end{align*}
As we have already see in the proof of Lemma \ref{lem3}  the tail $\P_z (\rho_{1}^{\hZ}>n-k)$ is supported by the first exit time from the cone so $\P_z (\rho_{1}^{\hZ}>n-k) \sim \P_z (\eta^{\hZ} >n-k)\leq C(n-k)^{-1/2}$ and as $k\leq n-\epsilon n$, 
\begin{align*}
    & \P(|\Sigma_{N^{\hZ}_n}-\cun (N^{\hZ}_n)^{(1- \alpha)^{-1}}| > \epsilon \cun (N^{\hZ}_n)^{(1- \alpha)^{-1}})  \\
    & \leq (\epsilon n)^{-1/2} \sum_{ j\in J_n}  \P(|\Sigma_{j}-\cun j^{(1- \alpha)^{-1}}| > \epsilon \cun j^{(1- \alpha)^{-1}}, \rho_j^{\hZ} \leq n) 
    \end{align*}
then by using the coupling with $Z$ (see Section \ref{coupling}), \eqref{id10}, and \eqref{Cheb}  we obtain that $\P(|\Sigma_{N_n^{\hZ}}-\cun (N_n^{\hZ})^{(1- \alpha)^{-1}}| > \epsilon \cun (N_n^{\hZ})^{(1- \alpha)^{-1}})\rightarrow 0$. \\
\noindent We are left to prove that in probability  $\Sigma_{N^{\hZ}_n}'$ is small comparing to $(N^{\hZ}_n)^{(1- \alpha)^{-1}}$, we start similarly as above : 
using Lemma \ref{lem1} and Corollary \ref{cor2} 
   \begin{align}
    & \P(|\Sigma_{N^{\hZ}_n}'|> \epsilon \cun (N^{\hZ}_n)^{(1- \alpha)^{-1}}) \nonumber \\
    &= \sum_{ j \in J_n} \sum_x \sum_{k \leq n-\epsilon n} \P( |\Sigma'_j | > \epsilon \cun j^{(1- \alpha)^{-1}}, \hX_{\rho^{\hZ}_{j}}=x, \rho^{\hZ}_{j}=k)\P_x (\rho^{\hZ}_{1}>n-k)+o(1) \nonumber \\
    & \leq \frac{1}{(\epsilon n)^{1/2}} \sum_{ j \in J_n }  \P( |\Sigma'_j|> \epsilon \cun j^{(1- \alpha)^{-1}})+o(1). \label{eq4}
    \end{align}  
The sum $|\Sigma'_j|=|\sum_{m=1}^{j}{\hX_{\eta_m^{\hZ}}-\hX_{\rho_{m-1}^{\hZ}}}|$ is stochastically dominated by  $\sum_{m=1}^{j}\hX_{\eta^{\hZ}}^{(m)}$ where $\hX_{\eta^{\hZ}}^{(m)}$ are independent copies of $\hX_{\eta^{\hZ}}$ under the probability distribution $\P(|\hZ_0=(0,1))$. Also let $0<\beta<1$, we obtain : 
\begin{align*}
     \P( \Sigma'_j> \epsilon \cun j^{(1- \alpha)^{-1}}) & = \P( |\sum_{m=1}^{j}\hX_{\eta^{\hZ}}^{(m)}| > \epsilon \cun j^{(1- \alpha)^{-1}})  \leq    \sum_{m=1}^{j}\E((\hX_{\eta^{\hZ}}^{(m)})^{\beta}) (\epsilon \cun j^{(1- \alpha)^{-1}})^{-\beta}.
\end{align*}  
Then, let us note that $\hX_{\eta^{\hZ}}$  is simply the simple random walk in the half-plane, stopped when it touches the horizontal axis. So a local limit result  for  for  $\hX_{\eta^{\hZ}}$ writes : for large $r$, $\P(\hX_{\eta^{\hZ}}=r|\hX_0=(0,1)) \sim cst\cdot r^{-2}$. This implies that $\E(|\hX_{\eta^{\hZ}}^{(m)}|^{\beta})$ is well bounded, so $\P( \Sigma'_j> \epsilon \cun j^{(1- \alpha)^{-1}}) \leq cst (\epsilon \cun)^{-\beta}j^{-\beta(1-\alpha)^{-1}+1}$. Then as $\alpha>0$, we can  choose a $\beta$ with our constraints such that  $-\beta(1-\alpha)^{-1}+1<0$. Replacing this upper bound in \eqref{eq4}, we finally obtain $ \P(\Sigma_{N_n^{\hZ}}'> \epsilon \cun (N_n)^{(1- \alpha)^{-1}}) \leq cst \cdot \epsilon^{-1-\beta}A n^{1/2(-\beta(1-\alpha)^{-1}+1)}$ which converges to zero.  \end{proof}

\noindent  Theorem \ref{the} is a consequence of Proposition \ref{prop3.4} together with Lemma \ref{lem1} and the comments that follow.


\section{The moments of $X_{\rho_i}$ and $X_{\eta_i}$ \label{secmoment}}

In the first section below, we study the moments of $\oZ_{\rho}$ (with  $\rho:=\inf\{m>0, X_m\cdot Y_m > 0  \}$), $\oZ_{\eta}$, and $\oZ_{\rho_1}$. Then, we apply these results to analyze the asymptotics of $\E(\oZ_{\rho_i})$ and $\E(\oZ_{\eta_i})$ as $i$ increases. Note that some of these estimates remain valid under the weaker condition $0<\alpha<1$, so we explicitly state the assumptions on $\alpha$ throughout this section.

\subsection{Elementary moments \label{Elenmoment}}

In this paragraph, we study the dominant coordinate of the walk after a single excursion.

\begin{lem}  \label{lem2.4} Assume $0<\alpha<1$, for any $z=(x,0)$ or $z=(0,x)$ with $x$ positive and large,
    \begin{align} 
    \E_z(\oZ_{\rho})= x + 2 x^{\alpha}-{\frac{3}{2}}-\frac{1}{6} \frac{1}{x^{\alpha}}+O\Big(\frac{1}{x^{ 2\alpha}}\Big) .\label{id6b} 
    \end{align}
For any $\beta \leq 2$, 
\begin{align} 
    \E_z(\oZ_{\rho}^{\beta}) = x^\beta &+2\beta x^{\beta-1+ \alpha}+O( x^{\beta-1}). \label{id6} 
\end{align}     
    
\end{lem}

\begin{proof} 
First, note that for any $z$, the expression of $\E_z(\oZ_{\rho})$ can be written explicitly; however, the hypothesis $x$ large simplifies its expression. For any $\beta$, $\E_z((\oZ_{\rho})^{\beta}) = \E_z((x + \rho-1)^{\beta})$. When $\beta = 1$, we just have to compute $\E_z(\rho)$, for which we use the equality $\P_z(\rho > k) = \prod_{m=0}^{k} \big(1 - \frac{1}{2(x + m)^{\alpha}}\big)$. Moreover, for large $x$, and $k \leq x^{\alpha + \epsilon}$ with $\epsilon > 0$ small enough so that $\alpha + \epsilon < 1$, 
$\prod_{m=0}^{k} \big(1 - \frac{1}{2(x + m)^{\alpha}}\big) = e^{-\frac{k}{2} \big( \frac{1}{x^{\alpha}} + \frac{1}{4x^{2\alpha}} + \frac{1}{12x^{3\alpha}} + O\big(\frac{1}{x^{4\alpha}}\big)\big)}$. This implies that 
\begin{align*}
   &  \E_z(\rho \un_{\rho \leq \ox^{\alpha+ \epsilon}})  \\
    & =\sum_{k \leq \ox^{\alpha+ \epsilon}} e^{-\frac{k}{2}(\frac{1}{\ox^{\alpha}} +\frac{1}{4 \ox^{ 2\alpha}}+ O(\frac{1}{\ox^{ 3\alpha}})) } = \Big({1-e^{-\frac{1}{2}(\frac{1}{\ox^{\alpha}} +\frac{1}{4 \ox^{ 2\alpha}}+\frac{1}{12 \ox^{ 3\alpha}} +O((\frac{1}{\ox^{ \alpha}})^4) }}\Big)^{-1}+O(\ox^{2\alpha}e^{-\ox^{\epsilon/2}}) \\
    & = 2 \ox^{\alpha}-\frac{1}{2}-\frac{1}{6} \frac{1}{\ox^{\alpha}}+O\big({\ox^{ -2\alpha}}\big)  .
\end{align*}
To finish we have to prove that  $\E_z(\rho \un_{\rho > \ox^{\alpha+ \epsilon}})$ is negligible comparing to the above expression, for that just note that $\prod_{m=0}^{k} \big(1-\frac{1}{2 (\ox+m)^{\alpha}}\big) \leq e^{-\frac{1}{2(1-\alpha)} (({\ox+k})^{1-\alpha}-\ox^{1-\alpha})}$, which implies that there exists a constant $c>0$ such that, $\E_x(\rho \un_{\rho > \ox^{\alpha+ \epsilon}}) \leq e^{-c (({\ox+\ox^{\alpha+ \epsilon}})^{1-\alpha}-\ox^{1-\alpha})} \leq e^{-\ox^{\epsilon/2}}$ so we obtain the result.
The moment of order $\beta$ can be treated quite similarly, indeed : 
\begin{align}
    \E_z((x+{\rho-1})^\beta)= \E_z((x+\rho-1)^\beta\un_{\rho \leq \ox^{\alpha+ \epsilon}})+\E_z((x+\rho-1)^\beta\un_{\rho > \ox^{\alpha+ \epsilon}}),
\end{align}
where like for the first moment just the first term count. So let us focus on this one :
\begin{align*}
    \E_z((x+{\rho-1})^\beta\un_{\rho \leq \ox^{\alpha+ \epsilon}})
   & = x^\beta \E_z((1+(\rho-1)/x )^\beta\un_{\rho \leq \ox^{\alpha+ \epsilon}})\\
   &=x^\beta \P_z(\rho \leq \ox^{\alpha+ \epsilon})+{\beta}{\ox^{\beta-1}} \E_z(\rho\un_{\rho \leq \ox^{\alpha+ \epsilon}} )+\frac{1}{2}\beta (\beta-1)\ox^{\beta-2} \E_z(\rho^2\un_{\rho \leq \ox^{\alpha+ \epsilon}} ) \\
   & +O(x^{\beta-1})+ O(\ox^{\beta-3} \E_z(\rho^3\un_{\rho \leq \ox^{\alpha+ \epsilon}} )).
\end{align*}
The first two terms have already been treated for the first moment. For the third one just note that  $\E_z(\rho^2\un_{\rho \leq \ox^{\alpha+ \epsilon}} )=-\ox^{\alpha+ \epsilon}\P_z(\rho>\ox^{\alpha+ \epsilon})+2\sum_{k=0}^{\ox^{\alpha+ \epsilon}}k\P_z(\rho>k)=8 \ox^{2\alpha}+O(1) $. So finally 
\begin{align*}
   & \E_z((x+\rho-1)^\beta\un_{\rho \leq \ox^{\alpha+ \epsilon}}) = \ox^\beta+2\beta \ox^{\beta-1+ \alpha} -\frac{\beta}{2} \ox^{\beta-1}-\frac{\beta}{6} \ox^{\beta-1- \alpha} +4\beta (\beta-1)\ox^{\beta-2+ 2\alpha}+O(\ox^{\beta-2+ \alpha})  .
\end{align*}
Finally for $\E_z((x+\rho-1)^\beta\un_{\rho > \ox^{\alpha+ \epsilon}})$, we prove similarly as for the first moment that the contribution is negligible.
\end{proof}

\noindent We now make the same analysis but for the random variable $Z_{\eta}$, where $\eta$ is the first instant the walk reach an  axis :  $\eta=\inf\{m>0, X_m\cdot Y_m=0\}$.

{
\begin{lem} \label{lemXetastar} Assume $z=(x,1)$ or $z=(1,x)$, there exists $c_3>0$ such that for large $x$,   
\begin{align}
\E_z(\oZ_{\eta})=\ox+c_3+O(\ox^{-1/2}) \label{id9b}
. 
\end{align}
similarly $\E_z(|\oZ_{\eta}-x|) \sim c_4$ with $c_4>0$. 
Moreover for any $\beta<1$, there exists $c_5>0$ such that for any large $\ox$,
\begin{align} 
    \E_z(\oZ_{\eta}^{\beta}) = \ox^{\beta}+c_5\ox^{\beta-1}+O(\ox^{\beta-3/2}). \label{id9}   
\end{align}
\end{lem}


\begin{proof} 
Assume $z=(x,1)$, computations are identical for the other case. Recall that $Z_{\eta}=(X_{\eta},Y_{\eta})$.
For the first statement, we initially assume $Y_{\eta} = 0$, and decompose:   $\E_z(\oZ_{\eta},Y_{\eta}=0)=x \P_x(1\leq  \oX_{\eta} \leq 2x-1)+ \sum_{   -x+1\leq k \leq x-1} k \P_z(\oX_{\eta}=x+k)+\sum_{ k \geq  x}  \P_z(\oX_{\eta}>x+k)$. The first term yields $x \P_z(1\leq  \oX_{\eta} \leq 2x-1)=x-x\P_z(\oX_{\eta} \geq 2x)=x-c+O(x^{-1})$
where $c$ is a positive constant, and we have used a small extension of the local limit result for $\P_x(X_{\eta}=y)$ proved in Lemma 5.1  in  \cite{AndDeb3}: 
 for any $\ell$ with $\ell \leq A x$ and $A>0$ there exists a positive bounded function $c(\ell/x)$ which converges when $x \rightarrow + \infty $  such that    
\begin{align}
\P_z(X_{\eta}>x+\ell ) \sim  \frac{x c(\ell/x)}{(x+ \ell)^2}.  \label{LLTX_eta1} 
\end{align} 
We also recall the following results presented in \cite{AndDeb3}, which are useful in the sequel
\begin{align}
 & \textrm{ if } \ell/x \rightarrow + \infty,\  \P_z(Z_{\eta}=(\ell,0) ) \sim  \frac{16}{\pi}\frac{x }{\ell^3},  \textrm{ if $\ell/x \rightarrow 0$},  \P_z(Z_{\eta}=(\ell,0) ) \sim  \frac{16}{\pi}\frac{ \max(\ell,1) }{x^3},  \label{LLTX_eta2} \\
& \textrm{ and for any } \ell>0,\    \P_z(Y_{\eta}> \ell ) \leq cst \ell^{-2}  \label{LLTX_eta3}.
\end{align} 
\noindent For the second term ($\sum_{   -x+1\leq k \leq x-1} k \P_z(\oX_{\eta}=x+k)$) we have by symmetry that for small $\eta$ (typically $\eta$ smaller than $ \epsilon x^2$, with $\epsilon>0$  small), the sum vanishes.  Otherwise, by using the above local limit result, we can prove that there exists a constant $c'>0$ such that $\sum_{   -x+1\leq k \leq x-1} k \P_z(\oX_{\eta}=x+k, \eta \geq \epsilon x^2)=c'+O(x^{-1})$. For the last term we use again the above local limit theorem, and get that there exists $c''>0$ such that $\sum_{ k \geq  x}  \P_z(\oX_{\eta}>x+k) = c''+O(x^{-1})$.
From this we deduce that $\E_z(\oZ_{\eta},Y_{\eta}=0)=x+c_3+O(x^{-1})$, with $c_3=c+c'+c''$, note that $c_3$ is necessarily positive as the walk remains on $(\Z_+)^2$, note however that the value of this constant has no importance at all in the sequel. \\
For $\E_z(\oZ_{\eta},X_{\eta}=0)$ we note that the random walk has to move from a coordinate $(x,1)$ to a coordinate $(0,.)$ which comes with a certain cost, again with the above local limit results, in particular \eqref{LLTX_eta2} and \eqref{LLTX_eta3},
\begin{align*}
\E_z(\oZ_{\eta},X_{\eta}=0)&=\sum_{k \leq x^{3/2}}\P_{z}(X_{\eta}=0,Y_{\eta}>k)+\sum_{k > x^{3/2}}\P_{z}(X_{\eta}=0,Y_{\eta}>k)\\
&\leq \sum_{k \leq x^{3/2}}\P_{z}(X_{\eta}=0)+\sum_{k > x^{3/2}}\P_{z}(X_{\eta}=0,Y_{\eta}>k) \\
& \leq \frac{cst}{x^{3/2}}+\sum_{k > x^{3/2}} cst \frac{x}{k^2} \leq  \frac{cst}{x^{1/2}}.
\end{align*}

\noindent For $\E_z(|\oZ_{\eta}-x|)$, just note that, $\E_z(|\oX_{\eta}-x|)=\sum_{k>0}k\P_z(X_{\eta}=x+k)-\sum_{k<0}k\P_z(X_{\eta}=x+k)$, and we proceed similarly as above. Also the case where $\oZ_{\eta}=Y_{\eta}$ under $\P_z$, leads to a negligible contribution.
\noindent For the second statement,
\begin{align} 
& \E_z(\oZ_{\eta}^{\beta}) \nonumber \\
& =\E_z(\oZ_{\eta}^{\beta} \un_{Y_{\eta}=0})+\E_z(\oZ_{\eta}^{\beta} \un_{X_{\eta}=0}) \nonumber \\
&=\sum_{ -x+1 <k \leq x} (\ox+k)^\beta \P_z(X_{\eta}=k+x)+\E_z(X_{\eta}^{\beta}\un_{X_{\eta}>2\ox} )+ \E_z(Y_{\eta}^{\beta} \un_{X_{\eta}=0}) \nonumber\\ 
&= \ox^\beta \P_z(X_{\eta} \leq 2x)+  \sum_{ k \leq x} ((\ox+k)^\beta-\ox^{\beta}) \P_z(X_{\eta}=\ox+k)+\E_z(X_{\eta}^{\beta}\un_{X_{\eta}>2\ox})+\E_z(Y_{\eta}^{\beta} \un_{X_{\eta}=0}). \nonumber
\end{align} 
The first term, which is equal to $\ox^\beta  -\ox^\beta \P_z(X_{\eta} > 2x) =\ox^\beta -c' \ox^{\beta-1}+o(\ox^{\beta-1})$, where the equality comes from \eqref{LLTX_eta1} for some constant $c'>0$.
We obtain also that the second term is of order $x^{\beta-1}$ again by using \eqref{LLTX_eta1}. The third term $\E_z(X_{\eta}^{\beta}\un_{X_{\eta}>2\ox})=o(x^{\beta-1})$ is treated similarly. The last term ($\E_z(Y_{\eta}^{\beta},X_{\eta}=0)$) is actually even smaller as the walk has to change of axis (moving from $(x,1)$ to something like $(0,\cdot)$ ). 
\end{proof}}

\begin{lem} \label{lemXrho_1} Assume $0<\alpha<1$, there exists $\ctr>0$ such that for any $z =(x,1)$ and $z=(1,x)$ with $x>0$ large 
    \begin{align} 
    \E_z(\oZ_{{\rho}_1})  = \ox + 2 \ox^{\alpha}+\ctr-\frac{1}{6}\frac{1}{\ox^{\alpha}}+o\Big(\frac{1}{\ox^{\alpha}}\Big)\label{id3}, 
    \end{align}  
moreover, for any $\beta<1$,    
  
\begin{align} 
    \E_z(\oZ_{\rho_1}^{\beta})  =& \ox^\beta+2\beta \ox^{\beta-1+ \alpha}+O(\ox^{\beta-1}). \label{id5} 
\end{align}

\end{lem}

\begin{proof}
Let $A>0$ large, strong Markov property, yields 
\begin{align*}
\E_z(\oZ_{\rho_1})&= \E_z( \E_{Z_{\eta}}(\oZ_{\rho})) =\E_z( \E_{Z_{\eta}}(\oZ_{\rho}) \un_{ \oZ_{\eta} \leq A})+\E_z( \E_{Z_{\eta}}(\oZ_{\rho}) \un_{ \oZ_{\eta} >A }).
\end{align*}
Following the idea of the proof of Lemma \ref{lem2.4}, we can get that for any $u$ such that, $\overline u \leq A$, $\E_u(Z_{\rho})\leq 2A $, so $\E_z( \E_{Z_{\eta}}(\oZ_{\rho}) \un_{ \oZ_{\eta} \leq A}) \leq 2 A \P_z(\oZ_{\eta} \leq A) \leq const \cdot A \ox^{-2} $, where the last inequality comes from \eqref{LLTX_eta2}.  For the other part we apply \eqref{id6b}, 
\begin{align*}
 & \E_z( \E_{Z_{\eta}}(\oZ_{\rho}) \un_{ \oZ_{\eta} >A }) =\E_{z}((\oZ_{\eta}+ 2(\oZ_{\eta})^{\alpha}-3/2-(\oZ_{\eta})^{-\alpha}/6+O((\oZ_{\eta})^{-2\alpha})) \un_{\oZ_{\eta} >A}) \\
&= \E_{z}((Z_{\eta}+ 2(\oZ_{\eta})^{\alpha}-3/2-(\oZ_{\eta})^{-\alpha}/6+O((\oZ_{\eta})^{-2\alpha})))\\ 
&  - \E_{z}((\oZ_{\eta}+ 2(\oZ_{\eta})^{\alpha}-3/2-(\oZ_{\eta})^{-\alpha}/6+O((\oZ_{\eta})^{-2\alpha})) \un_{\oZ_{\eta} \leq A}) \nonumber \\
&=\E_{z}(\oZ_{\eta}+2(\oZ_{\eta})^{\alpha}-3/2-(\oZ_{\eta})^{-\alpha}/6+O((\oZ_{\eta})^{-2\alpha})) + O(A \ox^{-2} ).
\end{align*} 
We are left to evaluate the above mean for which we can apply Lemma \ref{lemXetastar}. For \eqref{id5}, we write $\E_z(\oZ_{{\rho}_1}^{\beta})=\E_z( \E_{Z_{\eta}}( \oZ_{{\rho}}^{\beta}))$, then we apply \eqref{id6} and \eqref{id9}.

\end{proof}

\subsection{The asymptotic of the sequence $(\E(\oZ_{\rho_{i}}),i)$ \label{seqmoment}}
In this section we study the moments of $(\oZ_{\rho_i},i)$, we begin by discussing some basic facts about this sequence. 


\begin{lem} \label{lemsubmart} Assume $\alpha>0$, $(\oZ_{\rho_{i}},i)$ is a sub-martingale, moreover $\P(\lim_{i \rightarrow +\infty }\oZ_{\rho_i}=+\infty)=1$. 
\end{lem}

\begin{proof} Let $x>0$,
   \begin{align*} 
  \E(\oZ_{\rho_{i}}\un_{\oZ_{\rho_{i-1}}=x})= \E(\oZ_{\rho_{i}}\un_{Z_{\rho_{i-1}}=(x,1)})+\E(\oZ_{\rho_{i}}\un_{Z_{\rho_{i-1}}=(1,x)}) ,
 \end{align*}
by the strong Markov property 
\begin{align*}
& \E(\oZ_{\rho_{i}} \un_{Z_{\rho_{i-1}}=(x,1)}) \\ 
&=\E_{(x,1)}(\E_{Z_{\eta}}(\oZ_{\rho}) ) \P(Z_{\rho_{i-1}}=(x,1)) =\E_{(x,1)}(\oZ_{\eta}+\E_{Z_{\eta}}(\rho-1) ) \P(Z_{\rho_{i-1}}=(x,1)).
\end{align*}
By definition $Z_{\eta}$ dominates  stochastically the (same stopped) simple random walk on the half plane instead of the quarter plane. Moreover the mean, starting from $x$, of this stopped random on the half plane  is equal to $x$ by symmetry, this implies that $\E_{(x,1)}(Z_{\eta}) \geq x$. Also, as $\rho \geq 1$, $\E_{Z_{\eta}}(\rho-1) \geq 0$. So
$\E(\oZ_{\rho_{i}} \un_{Z_{\rho_{i-1}}=(x,1)}) \geq x \P(Z_{\rho_{i-1}}=(x,1))$; similarly $\E(\oZ_{\rho_{i}}\un_{Z_{\rho_{i-1}}=(1,x)}) \geq x \P(Z_{\rho_{i-1}}=(1,x))$. As $\P(\oZ_{\rho_{i-1}}=x)=\P(Z_{\rho_{i-1}}=(x,1))+\P(Z_{\rho_{i-1}}=(1,x))$,  this implies $\E(\oZ_{\rho_{i}}|\oZ_{\rho_{i-1}}=x) \geq x$, so $(\oZ_{\rho_{i}},i)$ is a sub-martingale.
Moreover the simple random walk on the quarter plane is almost surely not bounded, and it  crosses the axis infinitely many times almost surely, which implies that almost surely $\oZ_{\rho_{i}}$ diverges to $+\infty$ when $i\rightarrow + \infty$.
\end{proof}

In the Proposition below we prove that $\E_{.}(\oZ_{\rho_{i-1}})$ almost satisfies a recurrent equation.
\begin{prop} \label{prop5.4} Assume $0<\alpha<\nicefrac{1}{2}$, for any $z$ and large $i$
   \begin{align} 
  \E_{z}(\oZ_{\rho_{i-1}}) +2 (\E_{z}(\oZ_{\rho_{i-1}}))^{\alpha}+\ctr -a_{i-1} \leq \E_{z}(\oZ_{\rho_{i}})  \leq \E_{z}(\oZ_{\rho_{i-1}}) +2(\E_{z}(\oZ_{\rho_{i-1}}))^{\alpha}+\ctr,
   \label{id2} 
   \end{align}
where $(a_i,i)$ is a positive sequence which converges to zero when $i \rightarrow + \infty$ such that  $a_i (\E_z(\oZ_{\rho_{i-1}})^{2\alpha-1}\vee \E_z(\oZ_{\rho_{i-1}}^{-\alpha}))^{-1}$ is bounded. 
\end{prop}

\begin{proof} 
The strong Markov property yields $\E_z(Z_{\rho_{i}})= \E_z(\E_{Z_{\rho_{i-1}}}(\oZ_{\rho_1}))$,  then for large $i$, by Lemmata \ref{lemXrho_1} and \ref{lemsubmart}  
\begin{align*}
-\E_z(\oZ_{\rho_{i-1}}^{-\alpha}) \leq \E_z(\E_{Z_{\rho_{i-1}}}(\oZ_{\rho_1}))-\E_z( \oZ_{\rho_{i-1}}) -2\E_z( \oZ_{\rho_{i-1}}^{\alpha})-\ctr \leq 0.
\end{align*}

\noindent We are left to prove that $\E_z(\oZ_{\rho_{i-1}}^{\alpha})$ behaves almost like $\E_z(\oZ_{\rho_{i-1}})^{\alpha}$.
Note that there is nothing to do for the upper bound thanks to Jensen inequality, so the upper bound in \eqref{id2} is obtained. For the lower bound,  assume that for a certain index $i$ eventually large
\begin{align}
\E_z((\oZ_{\rho_{i-1}})^{\alpha})\geq (\E_z(\oZ_{\rho_{i-1}}))^{\alpha}-a_{i-1}.   
\end{align}
We now compute $\E_z((\oZ_{\rho_{i}})^{\alpha})$ and compare it with $(\E_z((\oZ_{\rho_{i}}))^{\alpha}$. First, we apply \eqref{id5} with $\beta=\alpha$,  and then the above assertion, 
\begin{align}
\E_z((\oZ_{\rho_{i}})^{\alpha})  &=\E_z(\E_{Z_{\rho_{i-1}}}(\oZ_{\rho_{1}}^{\alpha}))\geq \E_z(\oZ_{\rho_{i-1}}^{\alpha}) \geq (\E_z(\oZ_{\rho_{i-1}}))^{\alpha}-a_{i-1}. 
\label{reci_1}
\end{align}
Similarly, applying  \eqref{id3}
\begin{align}
\E_{z}(\oZ_{\rho_{i}}) =   \E_{z}(\E_{\oZ_{\rho_{i-1}}}(\oZ_{\rho_1}))
 \leq  \E_{z}(\oZ_{\rho_{i-1}})+2\E_{z}(\oZ_{\rho_{i-1}}^{\alpha})+\ctr. 
\end{align}
So an upper bound for $(\E_{z}(\oZ_{\rho_{i}}))^{\alpha}$, is given by 
\begin{align}
& (\E_{z}(\oZ_{\rho_{i}}))^{\alpha}  \leq  (\E_{z}(\oZ_{\rho_{i-1}}))^{\alpha}\Big(1+2\frac{\E_{z}(\oZ_{\rho_{i-1}}^{\alpha})}{\E_{z}(\oZ_{\rho_{i-1}})}+\frac{\ctr}{\E_{z}(\oZ_{\rho_{i-1}})} \Big)^{\alpha} \nonumber \\ 
&\leq  (\E_{z}(\oZ_{\rho_{i-1}}))^{\alpha}+(2+\ctr) \frac{(\E_{z}(\oZ_{\rho_{i-1}}))^{\alpha} \E_{z}(\oZ_{\rho_{i-1}}^{\alpha})}{\E_{z}(\oZ_{\rho_{i-1}})} \leq  (\E_{z}(\oZ_{\rho_{i-1}}))^{\alpha} +(2+\ctr) (\E_{z}(\oZ_{\rho_{i-1}}))^{2\alpha-1} \label{reci_2} 
\end{align}
where the last inequality follows from Jensen’s inequality. Now we subtract \eqref{reci_1} and \eqref{reci_2}, we obtain the lower bound :
\begin{align}
 & \E_z((\oZ_{\rho_{i}})^{\alpha}) -(\E_{z}(\oZ_{\rho_{i}}))^{\alpha}  \geq -a_{i-1}-(2+\ctr) (\E_{z}(\oZ_{\rho_{i-1}}))^{2\alpha-1}. \label{BinfZ} 
\end{align}
So this proves the lower bound of \eqref{id2} as by Lemma  \ref{lemsubmart}, $\oZ_{\rho_{i}}$ diverges almost surely toward $+\infty$. 
\end{proof}

\noindent From the above Proposition, we obtain two Corollaries, the first one concerns, in particular, the asymptotic in $i$ of the expectation of $\oZ_{\rho_i}$. 
\begin{cor} \label{Corid4} Assume $0<\alpha<\nicefrac{1}{2}$, for any fixed $\ox$, $z=(x,1)$ or $z=(1,x)$ and $i$ large 
   \begin{align} \E_{z}(\oZ_{\rho_i})  = \falf \cdot (i+\galf \cdot i^{\frac{1-2\alpha}{1-\alpha}})^{\frac{1}{1-\alpha}}+o_i(1)
    \label{id4}
   \end{align}
   where $\falf=(2(1-\alpha))^{\frac{1}{1-\alpha}}$ and $\galf=\frac{\ctr}{2(\falf)^{\alpha}}$.
   Also, 
   \begin{align}
    \sum_{m \leq i} \E_z(\oZ_{\rho_m}-\oZ_{\eta_m} )  \sim  \cun i^{1/(1-\alpha)}\ and\ \E_{z}(\oZ_{\eta_i}) \sim \cun i^{1/(1-\alpha)}. \label{id10}
    \end{align}
\end{cor}

\begin{proof}
\noindent For \eqref{id4}, let $(u_i,i)$ the sequence defined for any $i$ by $u_i=u_{i-1}+2(u_i)^{\alpha}+\ctr$, we first check that asymptotically in $i$, the sequence $v_i=\falf \cdot ( i+\galf\cdot i^{(1+2\alpha)/(1-\alpha)})^{1/(1-\alpha)}$  is asymptotically solution of this recurrent equation. By definition of $\falf$, and considering large $i$ 
\begin{align*}
    v_i&=\falf \cdot (i-1)^{\frac{1}{1-\alpha}} \Big( 1+ \frac{1}{i-1}+\galf \cdot \frac{i^{(1-2\alpha)/(1-\alpha)}}{i-1} \Big)^{1/(1-\alpha)} \\
    &= \falf \cdot (i-1)^{\frac{1}{1-\alpha}} + \frac{ \falf}{1-\alpha} \cdot (i-1)^{\frac{\alpha}{1-\alpha}} + \frac{ \falf \galf}{1-\alpha}i +\falf \galf \frac{ (1-2\alpha)}{(1-\alpha)^2}+O( i^{-1}) \\
    &= \falf \cdot (i-1)^{\frac{1}{1-\alpha}} + 2 (\falf) ^{\alpha} \cdot (i-1)^{\frac{\alpha}{1-\alpha}} + \frac{ \falf \galf}{1-\alpha}(i-1) +2(\falf)^{\alpha} \galf \frac{ (1-2\alpha)}{(1-\alpha)}+O( i^{-1}).
\end{align*}
Similarly 
\begin{align*}
    v_{i-1}&= \falf \cdot (i-1)^{\frac{1}{1-\alpha}}  \Big( 1+\galf \cdot \frac{(i-1)^{(1-2\alpha)/(1-\alpha)}}{i-1} \Big)^{1/(1-\alpha)} \\
    &= \falf \cdot (i-1)^{\frac{1}{1-\alpha}}  + \frac{ \falf \galf}{1-\alpha}(i-1)+O( i^{-1}),
\end{align*}
and
\begin{align*}
    (v_{i-1})^{\alpha}&= (\falf)^{\alpha} \cdot (i-1)^{\frac{\alpha}{1-\alpha}}  \Big( 1+\galf \cdot \frac{(i-1)^{(1-2\alpha)/(1-\alpha)}}{i-1} \Big)^{\alpha/(1-\alpha)} \\
    &= (\falf)^{\alpha} \cdot (i-1)^{\frac{\alpha}{1-\alpha}}  + \frac{ \alpha (\falf)^{\alpha} \galf}{1-\alpha}+O( i^{-1}).
\end{align*}
The last three estimations of $v_i$, $v_{i-1}$ and $(v_{i-1})^{\alpha}$ yields 
\begin{align*}
    v_i&=v_{i-1}+2 (v_{i-1})^{\alpha} +2(\falf)^{\alpha} \galf \frac{ (1-2\alpha)}{1-\alpha}+\frac{ 2\alpha (\falf)^{\alpha} \galf}{1-\alpha} +O( i^{-1}) \\
    &=v_{i-1}+2 (v_{i-1})^{\alpha} +2(\falf)^{\alpha} \galf  +O( i^{-1}).
\end{align*}
so asymptotically $v_i$ satisfies the recurrent equation of $u_i$ which gives the asymptotic of $u_i$. We then deduce \eqref{id4} using \eqref{id2}. \\
\noindent For \eqref{id10}, we have   $\E_z(\oZ_{\eta_j}-\oZ_{\rho_{j-1}})=\E_z(\E_{Z_{\rho_{j-1}}}(\oZ_{\eta})-\oZ_{\rho_{j-1}})\leq cste$, where the last inequality  comes from \eqref{id9b}. This implies that $\E_z( \sum_{j=1}^i(\oZ_{\eta_j}-\oZ_{\rho_{j-1}}))=o(i^{1/(1-\alpha)})$. Also by \eqref{id4},   $\E_z(\oZ_{\rho_i})\sim  \cun i^{1/(1-\alpha)}$, which proves the first equivalence, the second one is obtained similarly. 
\end{proof}

\noindent The second corollary below also deals with $\E_z(\oZ_{\rho_i})$, but with varying $z$ 

\begin{cor} \label{cor3} Assume $0<\alpha<\nicefrac{1}{2}$. Let $\epsilon>0$ small, let $i\geq 1$,   and $z\in \{(0,x),(1,x),$ $(x,0), (x,1)\}$, then for any $\ox \geq \epsilon^{-1} i^{(1-\alpha)^{-1}}$, $x^{\alpha} \leq    \E_{z}(\oZ_{\eta_i}^{\alpha}) \leq \E_z(\oZ_{\rho_i}^{\alpha} )  \leq  \ox^{\alpha}(1+ \epsilon)$. Otherwise if $\ox \leq \epsilon i^{(1-\alpha)^{-1}}$, and assume that $i$ is large then  $(\falf)^{\alpha} i^{\frac{\alpha}{1-\alpha}} \leq   \E_{z}(\oZ_{\eta_i}^{\alpha}) \leq \E_z(\oZ_{\rho_i}^{\alpha} ) \leq (\falf)^{\alpha} i^{\frac{\alpha}{1-\alpha}} (1+ \epsilon)$, finally if $\epsilon i^{(1-\alpha)^{-1}} \leq \ox \leq \epsilon^{-1} i^{(1-\alpha)^{-1}}$  there exists $c_3=c_3(\ox,\epsilon)>0$ which is bounded by a constant times $\epsilon^{-\frac{\alpha}{1-\alpha}}$ such that $ c_3 i^{\frac{\alpha}{1-\alpha}} \leq \E_{z}(\oZ_{\eta_i}^{\alpha}) \leq \E_z(\oZ_{\rho_i}^{\alpha}) \leq c_3 i^{\frac{\alpha}{1-\alpha}}(1+\epsilon)$. 
\end{cor}

\begin{proof}
\noindent \\
For the first statement, note that if $i=1$, this is just \eqref{id5}. Similarly $\E_z(\oZ_{\rho_2}^{\alpha})= \E_z( \E_{Z_{\rho_1}}( \oZ_{\rho_1}^{\alpha}))=\E_z(\oZ_{\rho_1}^{\alpha}+2\alpha \oZ_{\rho_1}^{2\alpha-1}+o(\oZ_{\rho_1}^{2\alpha-1}))$, then using  Lemma \ref{lemXrho_1} back again  we obtain $x^{\alpha} \leq \E_z(\oZ_{\rho_2}) \leq x^{\alpha}+o(x^{\alpha})$. We can then proceed recursively, the condition $\ox \geq \epsilon^{-1} i^{(1-\alpha)^{-1}}$ ensures that the sum of the terms which are added in the recurrence is still negligible compared to $x^{\alpha}$. 
The two other assertions come from Jensen's inequality, \eqref{BinfZ} and Corollary \ref{Corid4}.
\end{proof}

\subsection{The covariance of the sequence $(\oZ_{\rho_i}-\oZ_{\eta_{i}},i)$ \label{SecCov}}
In all this section we assume that $\alpha<1/2$.

\begin{prop} \label{procov}
   \begin{align}
     \lim_{i \rightarrow + \infty} i^{-(2\alpha+2)/(1-\alpha)} \E\Big( \big(\sum_{j\leq i} (\oZ_{\rho_j}-\oZ_{\eta_{j}})-\E(\oZ_{\rho_j}-\oZ_{\eta_{j}})  \big)^2\Big) 
    = 0  . \nonumber
\end{align}  
\end{prop}

\begin{proof}
The main contribution of the above second moment comes from the covariance which appears in the following upper bound 
\begin{align*}
   & \E\Big( \big(\sum_{j\leq i} (\oZ_{\rho_j}-\oZ_{\eta_{j}})-\E(\oZ_{\rho_j}-\oZ_{\eta_{j}})  \big)^2\Big)  \\
 &    \leq  2 \sum_{j \leq i} \sum_{k < j} |Cov(\oZ_{\rho_j}-\oZ_{\eta_{j}},\oZ_{\rho_k}-\oZ_{\eta_{k}})|  + \sum_{j \leq i} Var (\oZ_{\rho_j}-\oZ_{\eta_j}).  
\end{align*}
\noindent To obtain an upper bound for the covariance  we apply the following fact (proved in Lemma \ref{lemcov}  below) : for any $\epsilon>0$, large $k$ and  $j>k$ such that $j-k$ is large 
\begin{align*}
 & |Cov[(\oZ_{\rho_j}-\oZ_{\eta_{j}}),(\oZ_{\rho_k}-\oZ_{\eta_{k}})]| \\
 & \leq 8\E(\oZ_{\eta_{j-k}}^{\alpha})((Var(\oZ_{\eta_{k}}^{\alpha}))^{1/2}+\epsilon \E(\oZ_{\eta_k}^{\alpha}))+ 2Var( \oZ_{\eta_k}^{\alpha}) \nonumber  \\
 & +cste \epsilon^{-\alpha/(1-\alpha)}(j-k)^{\alpha/(1-\alpha)} ((Var( \oZ_{\eta_k}^{\alpha}))^{1/2}+\epsilon \E(\oZ_{\eta_k}^{\alpha}))+2\epsilon \E(\oZ_{\eta_k}^{2\alpha})+4 \E(\oZ_{\eta_k}^{\alpha}). 
 \end{align*}
From here, we need to estimate the mean and variance that appear on the right-hand side of the above inequality. To do so, we apply Jensen’s inequality, which implies that for any $0<\beta<1$, $\E(\oZ_{\eta_{\ell}}^{\beta}) \leq \E(\oZ_{\rho_{\ell}}^{\beta}) \leq  (\E(\oZ_{\rho_{\ell}}))^{\beta}$ then for sufficiently large $\ell$ by \eqref{id10}, $\E(\oZ_{\eta_{\ell}}^{\beta}) \leq 2\cun \ell^{\beta/(1-\alpha)} $. So we obtain an upper bound for the last three moments of the above inequality.
For the variance, as $\E(\oZ_{\rho_{k}}^{\alpha})=\E(\E_{Z_{\eta_{k}}}(\oZ_{\rho}^{\alpha}))$ by \eqref{id6} with $\beta=\alpha$ and $k$ large,  $\E(\oZ_{\eta_{k}}^{\alpha})\geq \E(\oZ_{\rho_{k}}^{\alpha})-1 $, also by \eqref{BinfZ},  $\E(\oZ_{\rho_k}^{\alpha}) \geq (\E(\oZ_{\rho_k}))^{\alpha}-a_k$ where $a_k$ converges to zero. As $\alpha<1/2$, this implies that, $Var( \oZ_{\eta_k}^{ \alpha}) \leq (\E(\oZ_{\eta_k}))^{2\alpha} -((\E(\oZ_{\eta_k}))^{\alpha}-a_k-1)^2 \leq c    \E(\oZ_{\rho_k})^{ \alpha} \leq  c' k^{\alpha/(1-\alpha)} $, with $0<c<c'$ are constants. \\
\noindent Finally, considering $Var (\oZ_{\rho_j}-\oZ_{\eta_j})$ which also appears in the second moment, we obtain for large $j$ : $Var (\oZ_{\rho_j}-\oZ_{\eta_j}) \leq c' j^{\alpha/(1-\alpha)}$. Then we can find a constant which is adjusted in order to encompass the terms with few contribution (that is bounded $j-k$ and $k$). So collecting all the upper bounds there exists two constants $C$ and $C'$ such that  
\begin{align*}
    & \E\Big( \big(\sum_{j\leq i} (\oZ_{\rho_j}-\oZ_{\rho_{j-1}})-\E(\oZ_{\rho_j}-\oZ_{\rho_{j-1}})  \big)^2\Big) \\
    &\leq C \epsilon^{\frac{1-2\alpha}{1-\alpha}}  \sum_{j \leq i } \sum_{k \leq j}  (j-k)^{ \alpha/(1-\alpha)}k^{\alpha/(1-\alpha)}  + \sum_{j \leq i } C' j^{\alpha/(1-\alpha)}  \leq C \epsilon^{\frac{1-2\alpha}{1-\alpha}} i^{(2+2\alpha)/(1-\alpha)},
\end{align*}
for the last inequality, $C$ has been adjusted accordingly. This finishes the proof as $\alpha<1/2$.
\end{proof}

We finish this section with the Lemma on the covariance that we used in the previous proof.

\begin{lem} \label{lemcov} For any $\epsilon>0$, large $k$ and $j>k$ such that  $j-k$ large 
\begin{align*}
 & |Cov[(\oZ_{\rho_j}-\oZ_{\eta_{j}}),(\oZ_{\rho_k}-\oZ_{\eta_{k}})]| \\
 & \leq 8\E(\oZ_{\eta_{j-k}}^{\alpha})((Var(\oZ_{\eta_{k}}^{\alpha}))^{1/2}+\epsilon \E(\oZ_{\eta_k}^{\alpha}))+ 2Var( \oZ_{\eta_k}^{\alpha}) \nonumber  \\
 & +cste\epsilon^{-\alpha/(1-\alpha)} (j-k)^{\alpha/(1-\alpha)} ((Var( \oZ_{\eta_k}^{\alpha}))^{1/2}+\epsilon \E(\oZ_{\eta_k}^{\alpha}))+2\epsilon \E(\oZ_{\eta_k}^{2\alpha})+4 \E(\oZ_{\eta_k}^{\alpha}). 
 \end{align*}

\end{lem}


\begin{proof}

\noindent The strong Markov property gives,
\begin{align*}
\E(\oZ_{\rho_j}-\oZ_{\eta_j}|Z_{\rho_k}) =\E_{Z_{\rho_k}}(\oZ_{\rho_{j-k}}-\oZ_{\eta_{j-k}}) = \E_{Z_{\rho_k}}(\E_{Z_{\eta_{j-k}}}(\oZ_{\rho})-\oZ_{\eta_{j-k}}).
\end{align*}
As $j-k$ is large (which implies that $\oZ_{\eta_{j-k}}$ is large almost surely by Lemma \ref{lemsubmart}), applying  \eqref{id3}, 
\begin{align*}
\E(\oZ_{\rho_j}-\oZ_{\eta_j}|Z_{\rho_k})= 2\E_{Z_{\rho_k}}(\oZ_{\eta_{j-k}}^{\alpha})+\ctr +o(1).
\end{align*}
 \noindent So we deduce that 
\begin{align*}
 \E((\oZ_{\rho_j}-\oZ_{\eta_j})(\oZ_{\rho_k}-\oZ_{\eta_k})) = \E[(2\E_{Z_{\rho_k}}( \oZ_{\eta_{j-k}}^{\alpha})+\ctr+o(1))(\oZ_{\rho_k}-\oZ_{\eta_k})].
\end{align*}
The expression of  $\E_{Z_{\rho_k}}(\oZ_{\eta_{j-k}}^{\alpha})$  depends whether or not   $\oZ_{\eta_k}\leq \oZ_{\rho_k}$ is large comparing to $(j-k)^{1/(1-\alpha)}$ : let $\epsilon>0$ small, by Corollary \ref{cor3} 
\begin{align*}
\E_{Z_{\rho_k}}(\oZ_{\eta_{j-k}}^{\alpha})&= \E_{Z_{\rho_k}}(\oZ_{\eta_{j-k}}^{\alpha})\un_{\Zeea{k} < \epsilon(j-k)^{1/(1-\alpha)}} +\E_{Z_{\rho_k}}(\oZ_{\eta_{j-k}}^{\alpha})\un_{\Zeea{k} > \epsilon^{-1}(j-k)^{1/(1-\alpha)}} \\
&+\E_{Z_{\rho_k}}(\oZ_{\eta_{j-k}}^{\alpha})\un_{\epsilon(j-k)^{1/(1-\alpha)} \leq \Zeea{k} \leq  \epsilon^{-1}(j-k)^{1/(1-\alpha)}}\\
& \leq (1+\epsilon) \E(\oZ_{\eta_{j-k}}^{\alpha})\un_{\Zeea{k} < \epsilon(j-k)^{1/(1-\alpha)}}+  (1+\epsilon)\oZ_{\rho_k}^{\alpha} \un_{\Zeea{k} > \epsilon^{-1}(j-k)^{1/(1-\alpha)}} \\
&+ (1+\epsilon)c_3 (j-k)^{\alpha/(1-\alpha)}\un_{\epsilon(j-k)^{1/(1-\alpha)} \leq \Zeea{k} \leq  \epsilon^{-1}(j-k)^{1/(1-\alpha)}},
\end{align*}
and recall that $c_3$ is bounded eventually depending on $\oZ_{\rho_k}$ and $\epsilon$. This implies, 
\begin{align}
 & \E((\oZ_{\rho_j}-\oZ_{\eta_j})(\oZ_{\rho_k}-\oZ_{\eta_k})) \nonumber \\ & \leq   (2\E(\Zeea{j-k}^{\alpha})(1+\epsilon)+\ctr+o(1))\E((\oZ_{\rho_k}-\oZ_{\eta_k})\un_{\Zeea{k} < \epsilon(j-k)^{1/(1-\alpha)}}) \nonumber\\
 & +\E((2 \oZ_{\rho_k}^{\alpha}(1+\epsilon)+\ctr+o(1))(\oZ_{\rho_k}-\oZ_{\eta_k})\un_{\Zeea{k} > \epsilon^{-1}(j-k)^{1/(1-\alpha)}}) \nonumber \\ 
 & +\E((2 c_3 (j-k)^{\alpha/(1-\alpha)} (1+\epsilon)+\ctr+o(1)) (\oZ_{\rho_k}-\oZ_{\eta_k})\un_{ \epsilon(j-k)^{1/(1-\alpha)}   \leq \Zeea{k}  \leq  \epsilon^{-1}(j-k)^{1/(1-\alpha)} }) \nonumber \\ &=:E_1+E_2+E_3 \label{LesE}.
\end{align}
Similarly,  
\begin{align}
& \E(\oZ_{\rho_j}-\oZ_{\eta_j}) \E(\oZ_{\rho_k}-\oZ_{\eta_k})  \nonumber\\ &=\E (2\E_{Z_{\rho_k}}(\oZ_{\eta_{j-k}}^{\alpha})+\ctr +o(1))\E(\oZ_{\rho_k}-\oZ_{\eta_k}),  \nonumber\\
&\geq (2\E(\oZ_{\eta_{j-k}}^{\alpha})+\ctr+o(1))\E(\oZ_{\rho_k}-\oZ_{\eta_k})  \P(\Zeea{k} < \epsilon(j-k)^{1/(1-\alpha)})  \nonumber\\
&+ (2\E( \oZ_{\rho_k}^{\alpha} \un_{\Zeea{k} > \epsilon^{-1}(j-k)^{1/(1-\alpha)}})+\ctr\P_{Z_{\rho_k}}({\Zeea{k} > \epsilon^{-1}(j-k)^{1/(1-\alpha)}}) +o(1)) \E(\oZ_{\rho_k}-\oZ_{\eta_k})  \nonumber\\
& + \E( (2c_3(j-k)^{\alpha/(1-\alpha)}+\ctr+o(1)) \un_{ \epsilon(j-k)^{1/(1-\alpha)}   \leq \Zeea{k}  \leq  \epsilon^{-1}(j-k)^{1/(1-\alpha)} })  \E(\oZ_{\rho_k}-\oZ_{\eta_k})  \nonumber \\ 
&=:E_1'+E_2'+E_3' \label{LesEp}.
\end{align}

\noindent We now subtract term by term \eqref{LesE} and \eqref{LesEp}, 
\begin{align*}
    E_1-E_1' = & (2\E(\Zea{j-k}^{\alpha})(1+\epsilon)+\ctr+o(1))\E((\oZ_{\rho_k}-\oZ_{\eta_k})\un_{\Zeea{k} < \epsilon(j-k)^{1/(1-\alpha)}}) \\
    &- (2\E(\oZ_{\eta_{j-k}}^{\alpha})+\ctr+o(1))\E(\oZ_{\rho_k}-\oZ_{\eta_k})  \P(\Zeea{k} < \epsilon(j-k)^{1/(1-\alpha)}) \\
     \leq & 2\E(\Zea{j-k}^{\alpha})  (\E((\oZ_{\rho_k}-\oZ_{\eta_k})\un_{\Zeea{k} < \epsilon(j-k)^{1/(1-\alpha)}})-\E(\oZ_{\rho_k}-\oZ_{\eta_k}) \P(\Zeea{k} <  \epsilon(j-k)^{1/(1-\alpha)})) \\ 
    & +2 \ctr\E(\oZ_{\rho_k}-\oZ_{\eta_k}) 
     + 2\epsilon \E(\Zea{j-k}^{\alpha})\E((\oZ_{\rho_k}-\oZ_{\eta_k})\un_{\Zeea{k} < \epsilon(j-k)^{1/(1-\alpha)}})   . 
\end{align*}    
We first address the first term on the right-hand side after the inequality. Conditioning with respect to $\oZ_{\eta_k}$, we use the Strong Markov property and the fact that under $\P_{(x,0)}$, $\oZ_{\rho}=x+\rho-1$ to obtain
\begin{align*}
& \E((\oZ_{\rho_k}-\oZ_{\eta_k})\un_{\Zeea{k} < \epsilon(j-k)^{1/(1-\alpha)}}|Z_{\eta_k}) \\
&=\E_{Z_{\eta_k}}( \oZ_{\rho})\un_{\Zeea{k} < \epsilon(j-k)^{1/(1-\alpha)}}-\oZ_{\eta_k} \un_{\Zeea{k} < \epsilon(j-k)^{1/(1-\alpha)}} \\
&= \oZ_{\eta_k} \un_{\Zeea{k} < \epsilon(j-k)^{1/(1-\alpha)}} + \E_{Z_{\eta_k}}( \rho-1)\un_{\Zeea{k} < \epsilon(j-k)^{1/(1-\alpha)}}-\oZ_{\eta_k} \un_{\Zeea{k} < \epsilon(j-k)^{1/(1-\alpha)}} \\
&=  \E_{Z_{\eta_k}}( \rho-1)\un_{\Zeea{k} < \epsilon(j-k)^{1/(1-\alpha)}},
\end{align*}
 with the same argument 
\begin{align*}
\E(\oZ_{\rho_k}-\oZ_{\eta_k}|Z_{\eta_k}) \P(\Zeea{k} <  \epsilon(j-k)^{1/(1-\alpha)})|Z_{\eta_k})=\E_{Z_{\eta_k}}( \rho-1)\P(\oZ_{\eta_k} < \epsilon(j-k)^{1/(1-\alpha)}),
\end{align*}
finally we use that  $\E_{Z_{\eta_k}}( \rho)=2Z_{\eta_k}^{\alpha}-1/2+o(1)$ (see the proof of Lemma \ref{lem2.4}).  Subtracting the last two equality and taking the mean 
\begin{align*}
& |\E(\E((\oZ_{\rho_k}-\oZ_{\eta_k})\un_{\Zeea{k} < \epsilon(j-k)^{1/(1-\alpha)}}|Z_{\eta_k})-\E(\oZ_{\rho_k}-\oZ_{\eta_k}|Z_{\eta_k}) \P(\Zeea{k} <  \epsilon(j-k)^{1/(1-\alpha)})|Z_{\eta_k}))| \\
&=| \E( (2Z_{\eta_k}^{\alpha}-3/2+o(1))    \un_{\Zeea{k} <  \epsilon(j-k)^{1/(1-\alpha)}}))
-\E( 2Z_{\eta_k}^{\alpha}-3/2+o(1))\P(\Zeea{k} <  \epsilon(j-k)^{1/(1-\alpha)}))|
\\
& \leq 4 (Var(\oZ_{\eta_{k}}^{\alpha}))^{1/2}
\end{align*}
where the inequality comes from the Cauchy-Schwarz inequality. For the last two  terms in the inequality for  $E_1-E_1'$, we just use that $\E(\oZ_{\rho_k}-\oZ_{\eta_k}) \leq  2 \E(\oZ_{\eta_k}^{\alpha})$ by \eqref{id6b}.  
So we obtain the following upper bound : 
\begin{align*}    
 |E_1-E_1'| 
& \leq  8\E(\oZ_{\eta_{j-k}}^{\alpha})((Var(\oZ_{\eta_{k}}^\alpha))^{1/2}+ \epsilon \E(\oZ_{\eta_k}^{\alpha})).
\end{align*}
We proceed similarly for $E_2-E_2'$, 
\begin{align*}
 |E_2-E_2'| 
& = |2\E(\oZ_{\rho_k}^{\alpha}(1+ \epsilon)\un_{\Zeea{k} > \epsilon^{-1}(j-k)^{1/(1-\alpha)}}  (\oZ_{\rho_k}-\oZ_{\eta_k}) )\\ 
& -2\E( \oZ_{\rho_k}^{\alpha}\un_{\Zeea{k} > \epsilon^{-1}(j-k)^{1/(1-\alpha)}} )   \E(\oZ_{\rho_k}-\oZ_{\eta_k})+2 \E(\oZ_{\eta_k}^{\alpha})| \\
& \leq 2|\E((\oZ_{\eta_k}^{\alpha}-\E(\oZ_{\eta_k}^{\alpha}))(Z_{\eta_k}^{\alpha}\un_{\Zeea{k} > \epsilon^{-1}(j-k)^{1/(1-\alpha)}}-\E(Z_{\eta_k}^{\alpha}\un_{\Zeea{k} > \epsilon^{-1}(j-k)^{1/(1-\alpha)}})))|  \\ 
&+2 \E(\oZ_{\eta_k}^{\alpha}) 
 + 2\epsilon \E(\oZ_{\eta_k}^{2\alpha} )  \leq 2( Var(\oZ_{\eta_{k}}^\alpha) + \E(\oZ_{\eta_k}^{\alpha}) + \epsilon \E(\oZ_{\eta_k}^{2\alpha} )). 
\end{align*}
and finally for $|E_3-E_3'|$,
\begin{align*}
 |E_3-E_3'| & \leq  2(j-k)^{\alpha/(1-\alpha)}(1+\epsilon)| \E(c_3  (\oZ_{\rho_k}-\oZ_{\eta_k})\un_{ \epsilon(j-k)^{1/(1-\alpha)}   \leq \o\oZ_{\eta_{k}}  \leq  \epsilon^{-1}(j-k)^{1/(1-\alpha)} }) \\
 & -\E(c_3\un_{ \epsilon(j-k)^{1/(1-\alpha)}   \leq \o\oZ_{\eta_{k}}  \leq  \epsilon^{-1}(j-k)^{1/(1-\alpha)} } ) \E(\oZ_{\rho_k}-\oZ_{\eta_k}) +2 \E(\oZ_{\eta_k}^{\alpha})| \\
 &\leq C \epsilon^{-\frac{\alpha}{1-\alpha}} (j-k)^{\alpha/(1-\alpha)} ((Var( \oZ_{\eta_k}^{\alpha}))^{1/2}+\epsilon \E( \oZ_{\eta_k}^{\alpha})) +2 \E(\oZ_{\eta_k}^{\alpha}) 
\end{align*}
for some constant $C>0$.
\end{proof}

\section{The cases $\alpha \geq 1$ and $\nicefrac{1}{2} \leq \alpha < 1$, beyond the first quadrant, and a last remark.  \label{sec5}}

{\bf The cases $\alpha \geq 1$ and $\nicefrac{1}{2}  \leq \alpha < 1$} \\
There is very few things to say about the case $\alpha\geq 1$ : as we have already seen in Lemma \ref{lem2.4}, for any $z$, $\E_z(\oZ_{\rho})=\ox+\E_x(\rho-1)$ and as  $\P_x(\rho>k)=\prod_{m=0}^{k} \Big(1-\frac{1}{2 (\ox+m)^{\alpha}}\Big)$, this implies (with $\alpha \geq 1$) that $\E_z(\oZ_{\rho})=+ \infty$. Since the walk within the cone is recurrent (and therefore crosses the axes at some point with probability one), this simple observation implies that the walk is, of course, transient and that almost surely, $\lim_{n \rightarrow + \infty} \frac{\oZ_n}{n}=1$. So, an interesting question (at least for $ \alpha = 1 $) could be: "What does the largest fluctuation on the cone observed before time $ n $ look like?". \\
Otherwise, the case $1/2 \leq \alpha <1 $ cannot be obtained with the method presented here, primarily because it is based on the analysis of the moment which can be  fully controlled only when $\alpha<1/2$.

\medskip
\noindent {\bf Beyond the first quadrant} \\
We consider the quarter plane here, but we can ask the same question by considering the whole plane, with the same hypothesis on each axis, that is, the walk is pushed away from the origin. In this case, it is not clear at all if our method works directly. This can be seen by looking at the recurrent equation satisfied by the mean of $\E_{z}(\oZ_{\rho_{i}})$; typically, we would have
\begin{align*} 
   \E_{z}(\oZ_{\rho_{i+1}})  \sim  \E_{z}(\oZ_{\rho_{i}}) +2\E_{z}(\text{sign}(\tilde Z_{\rho_i})|\oZ_{\rho_i}|^{\alpha})+\ctr,\ with\ \ctr \neq 0.
\end{align*}  
Here, $\text{sign}(\tilde Z_{\rho_i})$, which is the sign of the largest coordinate (in absolute value) between $X_{\rho_i}$ and $Y_{\rho_i}$ is, of course, what makes a difference compared to what we propose here. Then, there are two possible outcomes for the expected result: either the walk eventually chooses a quadrant and stays in it, which would yield the same result as ours, or there is some form of oscillation between quadrants. Note that  numerical simulations show that the first solution holds (see also Figure \ref{fig1}).

\begin{figure}[h] 
    \centering
    \includegraphics[width=0.7\textwidth]{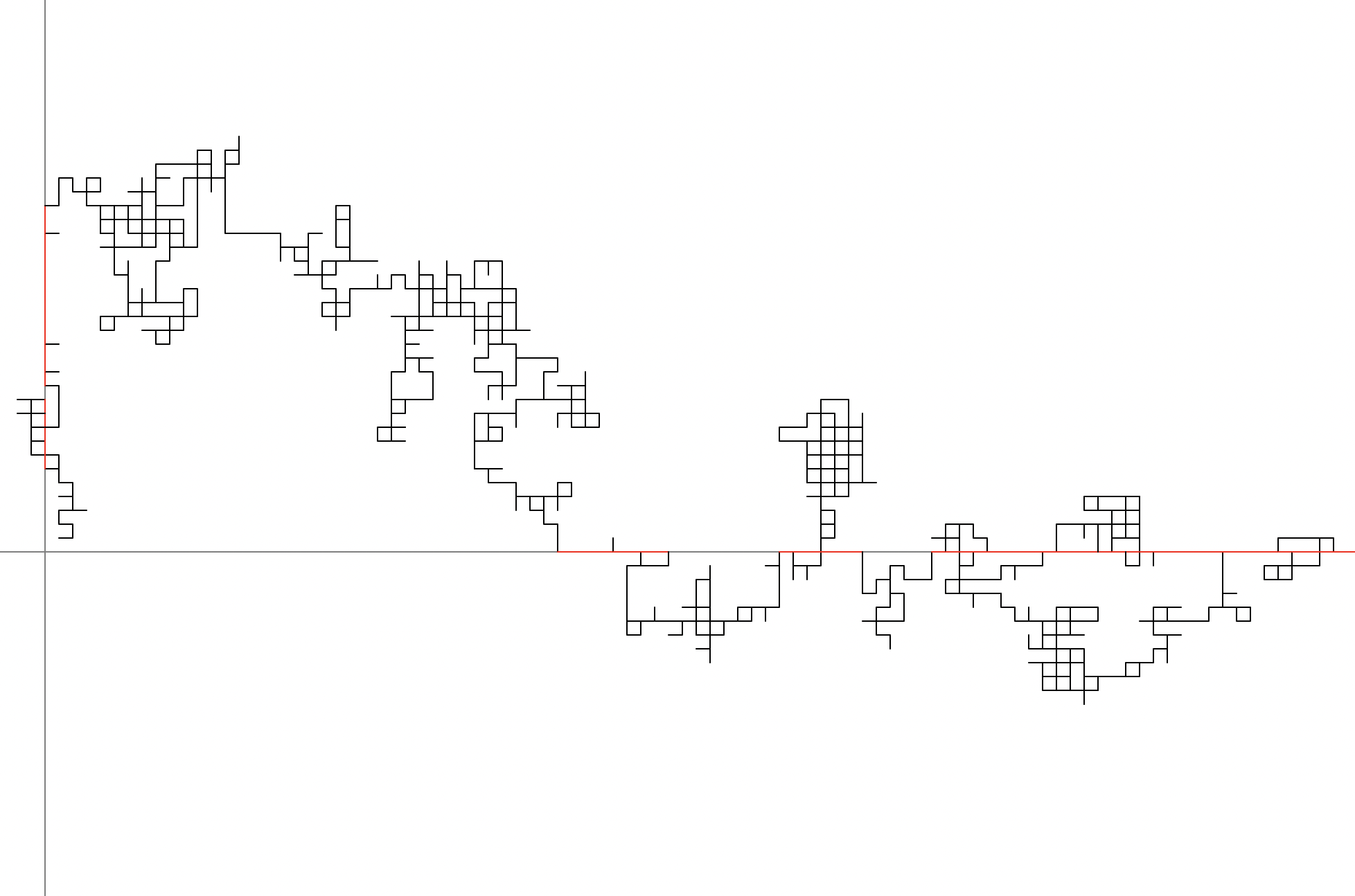}
    \caption{$Z$ on $\Z^2$, with $\alpha=0.2$}
    \label{fig1}
\end{figure}

\bigskip

\noindent {\bf Authorised the back-return on the axis} \\
We assume here that when on an axis the walk can not go backward, that is to say, toward the origin. In a hypothesis similar (but inverted) to that presented in  \cite{AndDeb3}, we can allow a backward movement on the axes (for example, starting from $(x, 0)$ with $x > 0$, we could  have $p((x, 0), (x-1, 0)) = \frac{1}{3} x^{-\alpha}$). Then, we can prove that the distribution of $X_{\rho}$ is such that there exists a positive bounded function ${ h} : \mathbb{N}^2 \rightarrow \mathbb{R}_+$ such that, for any $x$ large and and $y$ such that $0<y - \ox \leq A \ox^{\alpha}$,
\[
\P_{(x,0)}(\oX_{\rho}=(y,0)) \sim \frac{{ h}(x,y)}{ \oy^{\alpha}} e^{\frac{1}{2}(y^{1-\alpha} - \ox^{1-\alpha})}.
\]
But this implies that  this backward movement does not lead to any significant change, though it makes the computations more tedious.

\vspace{-0.3cm}
\section*{Acknowledgment} \vspace{-0.3cm}
I would like to thank Théo Ballu and Kilian Raschel for the very interesting discussions on the transient aspects of these walks.

\vspace{-0.3cm}
\bibliographystyle{alpha}
\bibliography{thbiblio}


\end{document}